\documentclass{report}
\usepackage[utf8]{inputenc}
\usepackage[english]{babel}
\usepackage{amsthm}
\usepackage[intlimits,sumlimits]{amsmath}
\usepackage{amssymb}
\usepackage{amsfonts}
\usepackage[hyperindex,breaklinks]{hyperref}
\usepackage{blindtext}
\usepackage[titletoc]{appendix}
\usepackage{authblk}

\usepackage{color}

\theoremstyle{plain}
\newtheorem{theorem}{Theorem}[chapter]
\newtheorem{lemma}[theorem]{Lemma}
\newtheorem{corollary}[theorem]{Corollary}
\newtheorem{proposition}[theorem]{Proposition}

\theoremstyle{definition}
\newtheorem{definition}[theorem]{Definition}
\newtheorem{remark}[theorem]{Remark}

\newcommand{\Adm}{\mathcal{A}dm}
\newcommand{\Admm}{\mathcal{A}dm_+}

\newcommand{\eps}{\varepsilon}

\newcommand{\heps}{\mathbb{H}}
\newcommand{\MM}{\mathbb{M}}
\newcommand{\diam}{\mathrm{diam}}
\newcommand{\scclass}{\mathcal{H}_{seq}}
\newcommand{\scfunclass}{\mathcal{H}}
\newcommand{\Subadd}{\mathrm{Subadd}}

\newcommand{\mdist}{\mathrm{Dist_m}}
\newcommand{\kdist}{\mathrm{Dist_K}}
\newcommand{\tro}{\tilde\rho}
\newcommand{\bro}{\bar\rho}

\newcommand{\sala}{\mathcal{A}}
\newcommand{\salb}{\mathcal{B}}
\newcommand{\aut}{\mathrm{Aut}}

\newcommand{\mdr}{\mathfrak{D}}
\newcommand{\mdrinf}{\mdr_\infty}
\newcommand{\Ury}{\mathbb{U}}

\newcommand{\MatrixDist}{\mathrm{mdist}}
\newcommand{\av}{\mathrm{av}}

\DeclareMathOperator{\essinf}{essinf}

\newcommand\blfootnote[1]{%
  \begingroup
  \renewcommand\thefootnote{}\footnote{#1}%
  \addtocounter{footnote}{-1}%
  \endgroup
}

\title{Dynamics of metrics in measure spaces and scaling entropy}

\author{A.~M.~Vershik, G.~A.~Veprev, P.~B.~Zatitskii}

\newcommand{\Addresses}{{
  \bigskip
 \footnotesize

  A.~M.~Vershik
  
  \textsc{St.~Petersburg Department of Steklov Institute of Mathematics, St.~Petersburg, Russia;}\par\nopagebreak
\textsc{St.~Petersburg State University,
St.~Petersburg, Russia;}\par\nopagebreak
\textsc{Institute for Information Transmission Problems, Moscow, Russia.}\par\nopagebreak
  \textit{E-mail}: \texttt{avershik@pdmi.ras.ru}

  \bigskip

    G.~A.~Veprev
    
    \textsc{St.~Petersburg State University,
St.~Petersburg, Russia;}\par\nopagebreak
\textsc{University of Geneva, Geneva, Switzerland.}\par\nopagebreak
  \textit{E-mail}: \texttt{georgii.veprev@gmail.com}

  \bigskip

  P.~B.~Zatitskii (corresponding author), 
  
  \textsc{St.~Petersburg Department of Steklov Institute of Mathematics, St.~Petersburg, Russia;}\par\nopagebreak
\textsc{University of Cincinnati, Cincinnati, OH, USA.
}\par\nopagebreak
  \textit{E-mail}: \texttt{pavelz@pdmi.ras.ru}

}}
\date{}

\begin{document}

\selectlanguage{english}

\maketitle
\newpage
\thispagestyle{empty}

{\bf Supported by an RSF grant (project 21-11-00152).}
\vfill
\Addresses

\newpage

\begin{abstract}
This survey is dedicated to a new direction in the theory of dynamical systems: the dynamics of metrics in measure spaces and new (catalytic) invariants of transformations with invariant measure. A space equipped with a measure and a metric naturally consistent with each other (a metric triple, or an $mm$\nobreakdash-space) automatically determines the notion of its entropy class, thus allowing one to construct a theory of scaling entropy for dynamical systems with invariant measure, which is different and more general compared to the Shannon--Kolmogorov theory. This possibility was hinted at by Shannon himself, but the hint went unnoticed. The classification of metric triples in terms of matrix distributions presented in this paper was proposed by M.~Gromov and A.~Vershik. We describe some corollaries obtained by applying this theory.

A brief overview of the paper is presented in the first chapter.


\end{abstract}

\setcounter{tocdepth}{1}
\tableofcontents

\chapter{The category of metric measure spaces. Historical overview and brief summary}
\label{sec1}
\section{Metric measure spaces}
\subsection{Measures and metrics: general considerations}

The joint consideration of a measure and a metric in one space has a long tradition. However, it was, of course,  preceded by a lengthy period of formation of the concepts of metric (and topological) spaces, metrization (F.~Hausdorff, P.~Ury\-sohn, and others), and the corresponding notions of measure spaces (A.~Lebes\-gue, J.~von~Neumann, A.~Kolmogorov, V.~Rokhlin, and others, 1900--1940).
In 1930--1950, both structures were already considered simultaneously (J.~Oxtoby, S.~Ulam, A.~D.~Alexandrov, and others), but still as quite separate entities. It is worth noting that some mathematicians, such as N.~Bourbaki (see his {\it Integration} volume, 1960s), held the view that measure spaces as a separate structure do not exist at all, and there are only various procedures for integrating functions. Such a one-sided position led N.~Bourbaki to ignore certain branches of mathematics, such as ergodic theory, the theory of $\sigma$-algebras, probability concepts, etc. For example, the lack of understanding that the theory of integration is one for all metric spaces hindered the development of  measure theory in functional spaces. On the other hand, many combinatorial constructions in measure theory and ergodic theory were not in demand and were not used in classical analysis due to the existence of a separating wall between them. Ignoring the structure itself and the category of measure spaces  is the reason why the most substantial and useful part of measure theory --- the geometry of various configurations of $\sigma$\nobreakdash-subalgebras (measurable partitions) --- remains little known and insufficiently developed.

\subsection{Metric triples and $mm$-spaces}

A new period began with the works of M.~Gromov, summarized in his book~\cite{Gro}. The book presents, in particular (Chapter~$3\frac{1}{2}$), 
a systematic study of the so-called
$mm$\nobreakdash-spaces, i.\,e., spaces equipped with both a metric and a measure.
An important viewpoint, expressed by M.~Gromov and simultaneously by A.~Vershik in~\cite{V02mccme,V02mp}, was as follows: in contrast to  the classical approach, which deals with various Borel measures on a~fixed complete metric space, they proposed to study various metrics on a fixed measure space (Lebesgue--Rokhlin space). In~\cite{Gro}, this approach was called ``reversed definition of mm spaces.'' This viewpoint was consistently pursued in~\cite{Vep20a, Vep20b, Vep21, Vep22, VPZ13a, VPZ14, Z14a, Z14b, Z15a,Z15b,PZ11, PZ15} and is presented in this survey. Namely, a theory of {\it
metric triples} $(X,\mu,\rho)$ --- space, measure, metric --- was constructed.
Here, the metric space $(X,\rho)$ was assumed to be complete and separable, while the measure space
$(X,\mu)$ was assumed to be a Lebesgue--Rokhlin space (in the case of continuous measure, it is isomorphic mod~0 to the interval $[0,1]$ with the Lebesgue measure, or to a~countable product of two-point spaces with the Haar measure), the structures of these spaces being naturally consistent with each other (for details, see Section~\ref{Sec_two_structures}). 
Gromov proved a classification theorem for such triples with respect to the group of measure-preserving isometries, while Vershik provided a version of this theorem with an explicit description of 
invariants of metric triples, namely, so-called {\it matrix distributions}, which are measures on the space of distance matrices. For more details, see~\cite{Gro, V98, V02ras} and Chapter~2 of this survey.

 One of the main advantages of this new focus was the ability to consider the {\it dynamics of metrics} with respect to the group of measure-preserving  automorphisms, which opened up a new source of invariants,  related to this dynamics,  for dynamical systems with invariant measure, such as {\it scaling entropy}. This is the primary topic of this survey; a brief summary of the results is provided in subsequent sections of the first chapter.
   
Transferring the center of gravity from metric to measure, on the one hand, simplifies the study of the measure-metric pair --- since it is well known that, up to a measure-preserving isomorphism, there exists a unique complete separable space with continuous measure defined on the whole $\sigma$-algebra of sets: namely, this is the unit interval with the Lebesgue measure, or, equivalently, a countable product of two-point spaces with the Haar measure. Therefore, in fact we may regard the entire possible arsenal of $mm$-structures as a set of metrics on a~universal measure space~$(X,\mu)$. However, we must first acknowledge that the metric~$\rho$ is merely a class of a.\,e.\ coinciding measurable functions of two variables satisfying the well-known axioms almost everywhere (i.\,e., mod~0). Such an object was called an {\it almost metric}. However, there is a
{\it correction theorem} (Theorem~\ref{th_ispr}), which allows us always to find a set of full measure on which this almost metric can be corrected to become a true semimetric. Here {\it the only consistency requirement for a metric and a measure  is that the metric is separable}, meaning that the $\sigma$-algebra of sets generated by all balls of positive radius is dense in the $\sigma$-algebra of mod~0 classes of measurable sets (see Theorem~\ref{ThRadonMeasure}). 
{\it Triples $(X, \mu, \rho)$ --- space, measure, metric --- satisfying the separability property  {\rm mod~0} are said to be admissible (or just metric triples)}; they are the object of study in the second chapter. In Theorems~\ref{Th_adm_cr} and~\ref{Matr_dop}, we present numerous equivalent formulations of the admissibility property. Among others, an important criterion for the admissibility of a triple $(X, \mu, \rho)$ is that the so-called $\eps$-entropy of this triple is finite for every positive~$\eps$, where the $\eps$-entropy is the logarithm of the number of balls of radius~$\eps$ that cover the entire space except for a set of measure at most~$\eps$.

Identifying  almost metrics coinciding a.\,e., {\it we may always assume that the metric space is complete} (i.\,e., in common terminology,  a Polish space). Indeed, if the space is not complete, we can consider its completion and extend the measure to this completion; the difference between the completion  and the original space will be of measure zero.

For metric triples, classical theorems usually stated under very modest assumptions can be generalized to very general statements. For instance, an elaboration of the well-known Luzin's theorem on the continuity of a measurable function is the following theorem: {\it any two admissible metrics are topologically equivalent (i.\,e., homeomorphic to each other) on a set of measure arbitrarily close to~$1$} (see Theorems \ref{th_two_metrics} and \ref{th_Lusin}).

\subsection{Classification of metric triples (and $mm$-spaces) and matrix distributions}
\label{sec113}

In Section~\ref{SecClass}, we will outline proofs of the classification theorem for metric triples, so here we provide only the precise formulation of the theorem about matrix distributions and their characterization. As for the possibility of classifying metric spaces, see~\cite{M}.

Consider the set of metric triples $(X,\mu,\rho)$ where $\mu$ is a nondegenerate continuous measure (i.\,e., its support coincides with~$X$). Recall that we assume that the metric space $(X,\rho)$ is complete and separable and $(X,\mu)$ is a Lebesgue space with continuous measure. 

The classification of $mm$-spaces with respect to measure-preserving isometries was provided by M.~Gromov~\cite{Gro} and A.~Vershik~\cite{V04}. In Gromov's formulation, a complete invariant is a set of naturally consistent random matrices  (for each positive integer~$n$, we consider the distances between $n$ points chosen randomly and independently according to the given distribution). Actually, the proof relies on the method of moments and the Weierstrass approximation theorem.
Vershik's proof involves the notion of the matrix distribution of a metric (or, more generally, of a measurable function of several variables), see below.

\begin{theorem}[matrix distribution as an invariant of metric triples]
A complete system of invariants of a metric triple $(X,\mu,\rho)$ with nondegenerate measure with respect to the group of all measure-preserving almost isometries is provided by the \textbf{matrix distribution} $\mdrinf = \mdrinf(X,\mu,\rho)$, that is, the probability measure on the space of infinite distance matrices that is the image of the Bernoulli measure $(X^{\infty},\mu^{\infty})$ under the map $\{x_i\}_i\mapsto \{\rho(x_i,x_j)\}_{i,j}$.
\end{theorem}

The proof of the theorem relies on the pointwise ergodic theorem and properties of completion of metric spaces.

In connection with this proof, the following question arises: how to describe matrix distributions as measures on the set of distance matrices? This question is discussed in detail in Section~\ref{Sec232}. It turns out that these measures can be described using a special notion of {\it simplicity of a measure}. This notion, as well as related considerations, is of general nature and  can be applied not only to the classification of metrics, but also to the classification of arbitrary measurable functions of several variables (see \cite{Ro,V02,XV1,XV}). Matrix distributions give rise to a whole range of problems in measure theory and learning theory; in particular,
 in this paper we consider the problem of reconstructing a metric and a measure in a space from randomized tests in which an important role is played by entropy and spectra.

\subsection{Spectral equivalence of metric triples}
\label{sec_spect}

Let us state an important problem with applications in learning theory and spectral graph theory which arises simultaneously with the definition of the matrix distribution. To the matrix distribution of an admissible metric we associate the collection of the spectra of the principal minors of random distance matrices. Recall that it is a system of interlacing sequences of real numbers 
$$
\{\lambda_1^n(\omega) \geq \lambda_2^n(\omega) \geq\dots \geq \lambda_n^n(\omega) \}, \qquad n=1,2,\dots.
$$
This system can be regarded as a random infinite triangular matrix. {\it Does this random matrix determine the original matrix distribution and, thus, the original metric?} A similar question is  to what extent the spectrum of an operator which naturally arises when considering a geometric object determines the object itself. The well-known Mark Kac's question ``Can one hear the shape of a drum?''  is precisely of  this kind: can a Riemannian manifold be uniquely recovered from the spectrum of the Laplace operator? 
In general, the answer to this question is negative. Another example from the early history of ergodic theory is whether the spectrum of the Koopman operator constructed from a~measure-preserving transformation determines the transformation itself?  The most meaningful negative answer is the discovery of Shannon--Kolmogorov entropy, which is a nonspectral invariant of transformations. In a sense, our question corresponds to these two examples, and the answer is likely to be negative as well. However, the situation in our case is entirely different. For example, the answer to a~similar question about measures on infinite symmetric (or Hermitian) matrices invariant under the orthogonal (unitary) group instead of the symmetric group is positive, because in this case the spectrum is a complete invariant even in the finite-dimensional case. Numerical experiments could be of use here. We  know only one experimental work~\cite{Bog}, done at the request of the first author, which showed a strong dependence of the set of random spectra on the dimensions of the spheres~$S^n$. It is of interest to study the nature of those metrics that are uniquely determined by the system of spectra; but in the general case, the answer is likely to be negative. It seems fruitful to study random spectra of matrix distributions (see~\cite{VP23}, as well as~\cite{KG00}).
On the other hand, there are many papers studying the spectra of distance matrices of specific metric spaces (and incidence matrices of graphs), see, e.\,g.,~\cite{HR22}.

\subsection{The cone of metric triples and the Urysohn universal  metric measure space}

The set of summable admissible metrics on a space with fixed measure is a~cone in the space $L^1(X^2,\mu^2)$ of functions of two variables, and we consider a~natural norm for which it is a complete normalized cone, see Section~\ref{sec214}. It is natural to study specific classes of metrics as subsets of this cone, and regard numerical characteristics of metric triples as functions on this cone. In particular, we regard $\eps$-entropy as a function on this cone or on bundles over it.

The properties of the map $(X,\mu,\rho) \rightarrow \mdrinf(X,\mu,\rho)$  from the cone of admissible metrics to the space of measures on distance matrices are described in detail in Section~\ref{SecClass}. It is appropriate to mention how the problems under consideration are related to the Urysohn  space. The Urysohn universal metric space $(\Ury,\rho_U)$ has become a popular object of study in recent years (after being forgotten for half a century). In \cite{V98, V02mccme}, a somewhat remarkable fact was proved:  in the space of all possible metrics on a countable set equipped with the weak topology, a dense $G_\delta$ subset consists of metrics whose completion yields a~space isomorphic to the Urysohn universal space. Apparently, this typicality of Urysohn spaces is preserved also in the context of admissible triples: {\it the collection of all (semi)metric triples $(X,\mu,\rho)$ for which the (semi)metric space $(X,\rho)$ is isometric to the Urysohn space is typical  in the weak topology of the cone of admissible metrics, i.\,e., it is a dense $G_\delta$ set in the space of all metric triples.} It follows that the set of metric triples such that, additionally, the measure is continuous and nondegenerate and the metric is Urysohn is also typical in the space of all metric triples. Nothing is known about Borel probability measures on the Urysohn  space, we do not know even characteristic examples of such measures; however, there is no doubt that they will be found in further research.

 \section{Metric invariants of dynamics in ergodic theory}
 
 Let us see how admissible metrics can be used in the theory of dynamical systems. Unfortunately, the impressive success of ergodic theory in the second half of the last century suffered from one drawback, which has only recently been truly felt: for various reasons, ergodic constructions rarely used metrics in phase spaces. Moreover, efforts were made to exclude metrics from consideration even in cases where their usefulness was evident. Nowadays, it has become clear that using metrics often allows one to define new measure invariants of dynamical systems. Namely, the metric is used in a nontrivial way in some construction, and the answer resulting from this construction is independent of the initial metric. The first example of such an invariant is scaling entropy, originally defined by A.~Vershik and further investigated by the authors in subsequent papers. Below we briefly explain the meaning of this invariant. It is described in detail in Chapter~\ref{sec3}, which contains a series of results obtained in recent years and constituting a new direction in ergodic theory.

\subsection{Scaling entropy of dynamical systems}

We explore the simplest possibility of using a metric. Namely, to an admissible metric we associate its usual $\eps$-entropy, i.\,e., a germ of a function in $\eps$. By averaging the metric, as is customary in the theory of transformations with invariant measures, we obtain a sequence of such germs of $\eps$-entropy. The main observation, stated as a conjecture in \cite{V10a, V10b} and proved in \cite{Z15a} using properties of metric triples, is that the asymptotics of the entropy (if it exists) as an equivalence class of growing (with  the number of averages) sequences of $\eps$-entropies does not depend on the choice of the initial metric, and thus this asymptotics becomes a new invariant of the dynamical system (see Section~\ref{sec311}). It is this invariant that was called the {\it 
scaling entropy of an automorphism} (see \cite{V10a, V10b}). In this definition, it was silently assumed that the equivalence class of sequences of $\eps$-entropies does not depend on $\eps$ provided that $\eps$ is sufficiently small. This is indeed the case in many examples.
For instance, the class $\{cn\}$ corresponds to positive Kolmogorov entropy, and the class of constant (in~$n$) sequences corresponds to discrete spectrum. As proved by Ferenczi and Park \cite{FP} and Zatitskiy~\cite{Z15b}, any equivalence class of growing sequences between these two monotone asymptotics can be realized for some automorphism (see Section~\ref{section_scaling_entropy_sequence_values}). However, only a few of them have been encountered so far. Computing the scaling entropy for specific automorphisms is a challenging task, which has been performed only for some examples.

\subsection{Comparison with classical entropy}

On the one hand, scaling entropy is a significant generalization of the notion of entropy and the Shannon--Kolmogorov entropy theory to cases where the Kolmogorov entropy is zero. However, the important issue here is the very notion of entropy of $mm$-spaces, i.\,e.,  of metric measure spaces. The presence of a~group of automorphisms suggests the idea of averaging metrics under the action of a~(e.\,g.,~amenable) group. The mention of C.~Shannon's name here is not coincidental. Apparently, over all these years, no one paid attention to and worked on deciphering Appendix 7 of the famous Shannon's paper \cite{Sh} on the foundations of information theory and its applications. In this appendix, Shannon suggests, in a very concrete and not immediately generalizable form,  the same idea that was formulated 60 years later (!) in \cite{V10a, V10b}: one should study the asymptotics of the {\it usual entropy of a metric measure space (i.\,e., the entropy of a metric triple)} for successive averages of the metric under an automorphism or a group of automorphisms. It appears that neither A.~N.~Kolmogorov (see~\cite{Kol}) nor his numerous followers paid enough attention to the fact that the entropy of an automorphism can be computed as the {\it asymptotic entropy of a metric measure space} (and the result does not depend on the metric, see Section~\ref{sec311}). It is true that Shannon, as A.~N.~Kolmogorov after him, was only interested in processes that transmit information (K-processes). The fact that the definition remains meaningful also for arbitrary automorphisms remained unnoticed until very recently. In our terms, Shannon considers metric triples that are finite fragments of a stationary process equipped with the Hamming semimetric; this is not always convenient, but the model of averages is universal. The subsequent development of the theory by Kolmogorov and his followers Rokhlin and Sinai somewhat concealed the generality of Shannon's idea, which is partly justified since the focus  at that time was on systems with positive entropy (hyperbolic, chaotic, etc.). It is worth noting that papers on {\it topological entropy}, which appeared shortly after those on metric entropy, would look much more natural within the framework of the theory of metric triples.

\subsection{Stability and instability}

In~\cite{Vep20a} (see also Section~\ref{section_nonstable_example}), an important step was taken in the study of scaling entropy: the realization that the answer to the question about the existence of a~universal equivalence class for the growth of entropies for averaged  metrics can be negative for specially constructed metric triples, i.\,e., the asymptotic behavior may significantly depend on $\eps$. Therefore, the definition of an equivalence class must involve functions of two variables: $n$ (the number of the average we take) and $\eps$. The final definition of scaling entropy involves such a~coarser entropy equivalence class for a given automorphism. See Section~\ref{sec311}. Presumably, automorphisms for which the asymptotics of entropies of averages depends on $\eps$ are typical. The definition given here appears to be the most general among all possible definitions related to the growth of entropies of automorphisms or groups of automorphisms (in the amenable case). It would be interesting to extend the theory of scaling entropy to nonamenable groups, see Section~\ref{section_groups}.
Note that similar generalizations of classical entropy theory have been proposed earlier, e.\,g., Kirillov--Kushnirenko entropy (sequential entropy, see~\cite{Kush}), Katok--Thouvenot slow entropy (see~\cite{KT}), Ferenczi's measure-theoretic complexity (see~\cite{F}). However, the theory of metric triples potentially contains also non-entropy invariants of dynamical systems; the key question is how to express them in terms of numerical invariants of metric triples, for instance, in terms of matrix distributions.

\subsection{Further use of metrics}

The introduction of scaling entropy is only the first step as concerns the use of metrics in ergodic theory. Moreover, this step could be achieved using a~sequence of Hamming semimetrics (and their averages) defined on cylinders of growing length in the symbolic presentation of the automorphism. Scaling entropy uses only the coarsest invariant of the metric~--- the asymptotics of the number of $\eps$-balls that almost cover the measure space. The next step should consist in using measures to study more complicated characteristics of sequences of metric spaces, and to single out those properties  that are invariants of an automorphism or  a group of automorphisms. Apparently, this problem has not been studied, and it is likely that behind such a study lies an intricate and interesting combinatorics of how continuous dynamics is approximated by finite constructions. Anyway, at first glance this approach significantly differs from the usual approximation theories.
The question about invariants of non-Bernoulli K-automorphisms discovered by D.~Ornstein remains unresolved for over 50 years. It is reasonable to assume that a geometric approach to the analysis of these invariants based on using metrics may help to advance in this direction. This is indicated by the characteristic suggested by the first author~--- the secondary entropy of a filtration (see~\cite{V15UMN}).

\subsection{Functions of several variables as a source of dynamic invariants}

Another general idea is that the familiar technique of functional analysis --- replacing the study of certain objects with the study of functions on these objects~--- has been used in the theory of dynamical systems in a very limited way (Koopman's idea): to a dynamical system
$\{T_g: g\in G\}$ one associates the group of operators $\{ U_g: f\mapsto U_g(f)(\,\cdot\,)=f(g^{-1}\,\cdot\,)\}$, where $f$ belongs to a~certain space of functions of one variable. This is the essence of the spectral theory of dynamical systems, which provides spectral invariants for the system. However, the same technique can be used for functions of several variables, for example, for the space of functions of two variables, but not arbitrary ones, but, say, metrics. Thus, we open up an entirely new way of constructing invariants of dynamical systems. In more detail, using the theory of admissible metric triples described in Chapter~\ref{sec2}, given a dynamical system with an invariant measure, we can associate to it a group of operators in the space of metric triples. Thus, earlier we actually embedded the group of measure-preserving automorphisms into the group of transformations of metric triples and verified that some invariants of triples (the asymptotics of the entropies of averages) do not depend on the metric and, therefore, become invariants of the original group of automorphisms of the measure space.

 On the other hand, using metrics allows one to introduce new notions into the development of certain classical results. One of  them is the tnotion of {\it ``virtual continuity'' of a measurable function of several variables}, which is related to the problem of  defining the restriction of such a~function  to elements of zero measure  of a measurable partition. This problem has an automatic positive solution (mod~0) for any measurable function of one variable, this is the meaning of  Rokhlin's theorem (see \cite{Ro1}). However, functions of two or more variables, in general, cannot be restricted to elements of a measurable partition. The restriction exists for so-called {\it virtually continuous functions}, whose definition is purely measure-theoretic (see \cite{VPZ14} and Definition~\ref{def_virt_cont} in Section~\ref{Sec_part}). In particular, any admissible metric is virtually continuous and, therefore, has such a restriction; thus, it turns almost every element of any measurable partition into an $mm$-space. Virtual continuity does not rely on any local properties of functions (such as smoothness etc.)\ and provides a new understanding of extension theorems for functions of several variables and Sobolev embedding theorems (see~\cite{VPZ13b}).

\subsection{General setting of the  metric isomorphism problem taking metric into account: catalytic invariants}
\label{section_catalyst}
Consider the isomorphism problem for measure-preserving automorphisms, i.\,e., the conjugacy problem in the group of classes of measure-preserving automorphisms coinciding mod~0. We will first consider it in a somewhat extended setting, namely, under the assumption that the group of automorphisms acts in a Lebesgue space~$X$ with continuous measure~$\mu$, on which we will also consider various admissible metrics~$\rho$. In this space, we fix an ergodic automorphism~$T$ and an admissible metric $\rho$. Let us consider a kind of partition function of metrics, more precisely, a normalized series in powers of  $z\in [0,1)$, and assume that it converges (literally or in a generalized sense):
 $$ \Omega_T(\rho,z)\equiv \Omega_T(z)=(1-z)\sum_{n=0}^{\infty} z^n \rho(T^n x, T^n y),\quad z\in [0,1). $$
Thus, we regard the function $\Omega_T(\,\cdot\,)$ of~$z$ as a metric that is a deformation of the metric
$\rho$ (for $z=0$) under the action of~$T$.

Note that for a fixed $z$, every term is the result of applying the operator $z\cdot U_T\otimes U_T$, acting on the cone of admissible metrics (on the space of functions of two variables). Correspondingly, the sum of the series  is the operator 
$$(Id - z\cdot U_T \otimes U_T)^{-1}.$$

Consider the function $\Omega_T(z)$ in a neighborhood of $z=1$. It is convenient to set $z=1-\delta$, then
 $$ \Omega_{T}(\rho,z)=\delta \sum_{n=0}^{\infty} (1-\delta)^n \rho(T^n x, T^n y), \quad 1>\delta \geq 0.$$

We will be interested in the behavior of the deformation $\Omega_T(\,\cdot\,)$  in a neighborhood of $\delta=0$, that is, $z=1$.

Assume that to every complete separable metric measure space $(X,\mu,\rho)$ we have associated a certain invariant (with respect to the isometries of the metric space)  $\Phi_{\rho}(\eps)$ whose values are  real functions of an argument common for all metric spaces, which we denote by $\eps$ (an example of  $\Phi_{\rho}(\eps)$ is the function  equal to the logarithm of the number of balls of radius $\eps$ that cover almost the whole space). Now we assume that there is a deformation  $\Omega_{T}(\rho,z)$ of metric spaces and consider the functions $\Phi_{\rho}(z,\eps)$ of two real variables $z\in [0,1)$ and $\eps>0$.
Finally, we introduce equivalence classes of functions of two variables depending on $\eps$ and on the deformation parameter $z$. Namely, first we form equivalence classes of functions of $z$ for a fixed $\eps$. If it turns out that such a class does not depend on $\eps$, then it is taken as an invariant. If the classes differ for different $\eps$, then we form coarser classes, combining all classes for different $\eps$. Anyway, these classes are actually associated with the family of metrics $\Omega_{T}(\rho,z)$ determined by the partition function $\Omega$. We say that the equivalence class of the function of two variables $\Phi_{\rho}(z,\varepsilon)$ is a {\it catalytic invariant} of the automorphism $T$ if it does not depend on the initial metric $\rho=\rho_0$ and thus is associated to the automorphism $T$ alone\footnote{The term ``catalytic'' (from ``catalyst'') is chosen because the definition of invariants of measure-preserving automorphisms is related to invariants of metrics, which are not part of the definition of automorphisms; however, the meaning of the invariant becomes clear if we use the metric, although, of course, it can be computed --- albeit quite intricately --- without using the metric.}. For more details, see~\cite{V23}.

In cases where the equivalence class depends on the initial metric, we obtain a~natural equivalence relation  on the initial metrics (considering classes of metrics with the same invariant) and can speak about a {\it relative catalytic invariant for a fixed class of initial metrics}. Relative invariants are also of interest for classification problem, which in this case is more detailed than the usual  measure-theoretic isomorphism problem.

The main question here is what invariants of the metric (with respect to the isometries) can be used as productively as the entropy of a metric measure space. This question remains open.

Another possibility for expanding the idea of catalytic invariants is to consider not only invariants of the metric space itself, but also {\it flag}-type invariants of metric spaces $X_1 \supset X_2 \supset ... \supset X_n$ in the same sense as above. There are certainly prospects here, but they require a more detailed study of metric spaces themselves and their invariants with respect to the isometries of $X_1$.

\chapter{Metric triples}
 \label{sec2}   
    \section{Metric triples and admissibility}
        \subsection{Measurable semimetrics, almost metrics, and correction theorems}
                             
            The classical approach to studying mm-spaces $(X, \mu, \rho)$, where $X$ is a space with measure $\mu$ and metric $\rho$, is typically based on investigating various Borel measures $\mu$ on a fixed metric space $(X,\rho)$. However, as mentioned in Chapter~\ref{sec1}, we consider the topic of metric measure spaces  from a relatively new perspective. This viewpoint appears to have been first introduced in the papers~\cite{V02mccme} and~\cite{V02mp}. We start with a fixed measure space $(X,\sala,\mu)$ (Lebesgue--Rokhlin space, a standard probability space with  continuous measure, i.\,e., isomorphic to the interval $[0,1]$ with the Lebesgue measure) and explore various metrics $\rho$ on this space. The condition that connects the topological and measurable structures is the \emph{measurability of the metric $\rho$} as a function of two variables; such metrics will be referred to as \emph{measurable}. Furthermore, we have to consider natural generalizations of metrics --— \emph{semimetrics}\footnote{A semimetric is a non-negative symmetric function of two variables that vanishes on the diagonal and satisfies the triangle inequality. In  literature, the term pseudometric is also commonly used for such functions.}.
          
            In the context of the chosen approach to studying metrics (and semimetrics) as measurable functions on $(X^2, \mu^2) = (X \times X, \mu \times \mu)$, the concept of {\it almost metric} naturally arises --- a function for which the defining metric relations are satisfied almost everywhere (not necessarily everywhere).

            \begin{definition}
                A measurable non-negative function $\rho$ on $(X^2,\mu^2)$ is called an \emph{almost metric} (or metric mod 0) on $(X,\mu)$ if
                \begin{enumerate}
                \item $\rho(x,y) = \rho(y,x)$ for $\mu^2$-almost all pairs $(x,y) \in X^2$;
                \item $\rho(x,z) \leq \rho(x,y) + \rho(y,z)$ for $\mu^3$-almost all triples $(x,y,z) \in X^3$.
                \end{enumerate}
            \end{definition}

            For instance, if a sequence of measurable metrics converges in measure or almost everywhere, the limit can a priori turn out to be only an almost metric. We present the following theorem on correction of almost metrics, see~\cite{PZ11}.

            \begin{theorem}[Correction Theorem]
            \label{th_ispr}
                If $\rho$ is an almost metric on $(X,\mu)$, then there exists a semimetric $\tilde \rho$ on $(X,\mu)$ such that the equality $\rho = \tilde \rho$ holds $\mu^2$-almost everywhere.
            \end{theorem}
            
            The proof of the correction theorem is based on two steps: identify the space $(X,\mu)$ with the circle $\mathbb{S}$ equipped with the Lebesgue measure $m$ and apply Lebesgue's differentiation theorem. For an almost metric $\rho$ on $(\mathbb{S},m)$, as its correction, one can take the function
            $$
            \tilde{\rho}(x,y) = \limsup_{T \to 0^+}\frac{1}{T^2}\int_0^T\int_0^T \rho(x+t,y+s)\,dt\,ds, \quad x,y \in \mathbb{S},
            $$
            which coincides with $\rho$ almost everywhere and is a semimetric.

            Theorem~\ref{th_ispr} allows us to further restrict our consideration to measurable semimetrics, rather than almost metrics. In particular, this theorem implies that the set of measurable semimetrics is closed with respect to convergence in measure and almost everywhere.

            The development of correction ideas for functions of several variables can be found in~\cite{P16}.

        \subsection{Admissibility. Relation between measurable and metric structures}
            \label{Sec_two_structures}
            In the future, we will mostly work with so-called admissible semimetrics and metrics. One of the definitions of admissibility is separability on a subset of full measure.

            \begin{definition}
            A measurable (semi)metric $\rho$ on a measure space $(X,\sala, \mu)$ is called \emph{admissible} if there exists a subset $X_0\subset X$ such that $\mu(X_0)=1$ and the (semi)metric space $(X_0,\rho)$ is separable. The triple $(X,\mu,\rho)$ will also be referred to as an \emph{admissible} (semi)metric triple, or simply a \emph{metric triple}.
            \end{definition}
            
            The relation between the structures of a measurable and a metric space on $X$ is illustrated by the following statements. An admissible metric on $(X, \sala, \mu)$ generates the Borel sigma-algebra $\salb$, and the measure $\mu$ turns out to be Borel.

            \begin{theorem}[see~\cite{VPZ14}]
            \label{ThRadonMeasure}
            If $\rho$ is an admissible metric on the space $(X,\sala,\mu)$, then the measure $\mu$ is a Radon measure on the metric space $(X,\rho)$. The Borel sigma-algebra $\salb$ generated by the metric $\rho$ on $X$ is a subalgebra of the original sigma-algebra $\sala$ and is dense in it.
            \end{theorem}
            The inclusion $\salb\subset \sala$ follows from the fact that open balls of radius $R$ are cross sections of the set $\{(x,y) \in X\times X\colon \rho(x,y)<R\}$, hence they belong to~$\sala$. Due to separability (mod 0), any open set can be represented as a countable union of balls, and therefore it also lies in $\sala$. The density of $\salb$ in $\sala$ arises from the maximal property of the Lebesgue sigma-algebra. For further details, we refer to~\cite{VPZ14}.

            A consequence of the Radon property of a measure is the following somewhat unexpected fact, that demonstrates in a sense the universality of an admissible metric triple.

            \begin{theorem}[see~\cite{VPZ14}]
            \label{th_two_metrics}
            If $\rho_1$ and $\rho_2$ are two admissible metrics on a measure space $(X,\mu)$, then for any $\eps>0$, there exists a subset $X_0 \subset X$ such that $\mu(X_0) > 1-\eps$ and the topologies induced by the metrics $\rho_1, \rho_2$ on $X_0$ coincide.
            \end{theorem}

            From this theorem, we can deduce the following generalized Luzin's theorem.

            \begin{theorem}[Generalized Luzin's Theorem]
            \label{th_Lusin}
            Let $(X,\mu,\rho)$ be a metric triple, and let $f$ be a measurable function on $(X,\mu)$. Then for any $\varepsilon>0$, there exists a subset $X_0 \subset X$ such that $\mu(X_0) > 1-\varepsilon$ and the function $f$ is continuous on $(X_0,\rho)$.
            \end{theorem}

            We mention that in~\cite{BKP} an alternative approach to these results is presented. Theorem~\ref{ThRadonMeasure} concerning the Radon property of a measure is derived from the generalized Luzin's theorem.

            We conclude the paragraph with the following remark. With an admissible semimetric $\rho$ we associate a partition $\xi_\rho$ into sets of zero diameter: points $x,y \in X$ belong to the same element of the partition $\xi_\rho$ if and only if $\rho(x,y) = 0$. Admissibility of the semimetric $\rho$ guarantees that the partition $\xi_\rho$ is measurable. Thus, an admissible semimetric $\rho$ induces an admissible metric on the quotient space $X/\xi_\rho$.

        \subsection{Metrics and partitions}
        \label{Sec_part}

    Let $\pi\colon(X,\mu)\to (Y,\nu)$ be a measurable mapping that transforms $\mu$ to $\nu$. Let~$\xi$ be a partition into preimages of points under $\pi$. A classical result by V. A. Rokhlin (see~\cite{Ro1}) states the existence and uniqueness (mod~0) of conditional measures~$\mu_y$ on the elements of $\pi^{-1}(y)$ for $\nu$-almost every $y \in Y$, such that the measure $\mu$ is the integral of the measures $\mu_y$ with respect to $\nu$: $\mu = \int_Y \mu_y\ d\nu(y)$. For a measurable function $f$ defined on the space $(X,\mu)$, a system of restrictions $f_y = f|_{\pi^{-1}(y)}$ on the spaces $(\pi^{-1}(y), \mu_y)$, $y \in Y$, can be associated. For $\nu$-almost every $y \in Y$, the restriction $f_y$ is a measurable function. When replacing the function $f$ with an equivalent function (coinciding almost everywhere with respect to $\mu$), the traces of $f_y$ for $\nu$-almost every $y\in Y$ are replaced with $\mu_y$-equivalent traces. Thus, the restrictions of a measurable function to the elements of a measurable partition are well-defined. A similar question for multi-variable functions was raised by A. Vershik: is it possible to consistently define the restrictions of a given almost everywhere defined multi-variable function to the elements of a measurable partition? The answer, in general, is negative.

    However, for the so-called virtually continuous multi-variable functions, it is possible to provide a consistent definition of restrictions to the elements of a measurable partition. We present one of the possible definitions for virtually continuous functions.

    \begin{definition} \label{def_virt_cont}
    A measurable function $f$ on $(X^2,\mu^2)$ is called \emph{properly virtually continuous} if there exists a subset $X'\subset X$ of full measure and an admissible metric $\rho$ on $(X,\mu)$ such that $f$ is continuous on $X'\times X'$ with respect to the metric $\rho\times \rho$. A measurable function $f$ on $(X^2,\mu^2)$ is called \emph{virtually continuous} if it coincides with some properly virtually continuous function $\mu^2$-almost everywhere.
    \end{definition}

    A simple example of a non-virtually continuous function is the indicator function of the set "above the diagonal": the function $\chi_{\{x<y\}}$ on $X^2$, where $X = [0,1]$.

    \begin{theorem}
    If two proper virtually continuous functions $f$ and $g$ on $(X^2,\mu^2)$ coincide $\mu^2$-almost everywhere, then there exists a subset $X_0\subset X$ of full measure such that $f$ and $g$ coincide on the square $X_0^2$.

    If $\xi$ is a partition into preimages of points under a measurable mapping $\pi\colon (X,\mu) \to (Y,\nu)$, and $\{\mu_y\}_{y \in Y}$ is the corresponding system of conditional measures, then for $\nu$-almost every $y \in Y$, the restrictions of $f$ and $g$ to $\pi^{-1}(y)\times \pi^{-1}(y)$ coincide $\mu_y\times \mu_y$-almost everywhere.
    \end{theorem}
    
    The presented theorem allows giving a proper definition of restrictions of a virtually continuous function as restrictions of its equivalent proper virtually continuous function.

    \begin{proposition}
    An admissible semimetric $\rho$ on a measure space $(X,\mu)$ is a properly virtually continuous function.
    \end{proposition}
    Indeed, an admissible semimetric $\rho$, considered as a function of two variables, is trivially continuous on $X\times X$ with respect to the semimetric $\rho\times \rho$. Therefore, restrictions of $\rho$ to the elements of a measurable partition $\xi$ are well-defined. Moreover, for $\nu$-almost every $y\in Y$, the restriction of the semimetric $\rho$ to $\pi^{-1}(y)$ becomes an admissible semimetric on $(\pi^{-1}(y),\mu_y)$.

    For more details on virtually continuous functions, we refer to~\cite{VPZ13b} and~\cite{VPZ14}. In these papers, relations between the concept of virtual continuity and classical questions of analysis are discussed, such as theorems regarding traces of functions in Sobolev spaces, nuclear operators and their kernels. Virtual continuity also appears in dual-type theorems like the Monge–Kantorovich duality theorems, see papers~\cite{VPZ14}, \cite{BKP}, and~\cite{B22}.

    \subsection{The cone of summable admissible semimetrics and m-norm}  
    \label{sec214}
    In the following, we will focus our study on \emph{summable admissible semimetrics}, i.\,e. admissible semimetrics $\rho$ with a finite integral
    $$
    \int\limits_{X\times X} \rho(x,y) \, d\mu(x) \, d \mu(y) < \infty.
    $$
    For a fixed measure space $(X,\mu)$, let $\Adm(X,\mu)$ denote the set of all summable admissible semimetrics on $(X,\mu)$. In Section~\ref{section_eps_entropy_definition}, we will show that this set forms a convex cone in $L^1(X^2,\mu^2)$. The group $\aut(X,\mu)$ of automorphisms of the space $(X,\mu)$ acts on the cone $\Adm(X,\mu)$ by translations. The dynamics of metrics studied in Chapter~\ref{sec3} is essentially the study of this action. Orbits of this action consist of pairwise isomorphic metrics. In the future, we will discuss these orbits and metric averages over them. In particular, we will investigate the behavior of invariants of automorphisms that arise when studying this action. To each metric $\rho$ from $\Adm(X,\mu)$ corresponds a stabilizer, which is a group of measure-preserving isometries mod~0 of the space $(X,\rho)$.
    
    The cone $\Adm(X,\mu)$ is not closed in $L^1(X^2,\mu^2)$, and its closure is the set  of all (not necessarily admissible) summable semimetrics. When working with admissible semimetrics, it is convenient to use another norm induced in $L^1$ by the cone of all summable semimetrics. This norm was introduced in~\cite{VPZ13a}.
    
    \begin{definition}\label{def_mnorm}
    For a function $f \in L^1(X^2,\mu^2)$ we define its  m-norm (finite or infinite) as follows:
    $$
    \|f\|_{m} = \inf \left\{\|\rho\|_{L^1(X^2,\mu^2)}\colon \rho \text{ is a semimetric on } (X,\mu), \ |f| \leq \rho \ \ \mu^2\text{-a.\,e.}\right\}.
    $$
    We denote by $\MM(X,\mu)$ the subspace of $L^1(X^2,\mu^2)$ that contains the functions with a finite m-norm.
    \end{definition}
    Clearly, the m-norm dominates the standard norm in $L^1(X^2,\mu^2)$, implying that convergence in the m-norm implies convergence in $L^1(X^2,\mu^2)$.
    
    It is also evident that $\Adm(X,\mu) \subset \MM(X,\mu)$. In Section~\ref{sec_convergenceinmnorm}, we will discuss the properties of the cone of admissible metrics and the m-norm. We will show that the space $\MM(X,\mu)$ and the cone $\Adm(X,\mu)$ are complete with respect to the m-norm.
    
    Alongside the cone $\Adm(X,\mu)$, we will also consider its subset --- the cone $\Admm(X,\mu)$, consisting of summable admissible metrics. It forms a dense subset in $\Adm(X,\mu)$.

    \section{Epsilon entropy of a metric triple}
        \subsection{Epsilon entropy and characterization of admissibility} \label{section_eps_entropy_definition}

        One of the simplest functional characteristics describing a metric triple is the notion of epsilon-entropy that goes back to Shannon (see Chapter~\ref{sec1}).
\begin{definition}\label{definition_epsilon_entropy}
Let $\rho$ be a measurable semimetric on $(X,\mu)$ and $\eps>0$. The $\eps$-entropy $\heps_\eps(X,\mu,\rho)$ of the semimetric triple $(X,\mu,\rho)$ is defined as $\log k$, where $k$ is the minimal number (or infinity) for which the space $X$ can be represented as a union $X = X_0 \cup X_1\cup \dots \cup X_k$ of measurable sets, such that $\mu(X_0)<\eps$ and $\diam_{\rho}(X_j)<\eps $ for all $j = 1,\dots, k$. For $\eps\geq 1$, we set $\heps_\eps(X,\mu,\rho)=0$.
\end{definition}
Sometimes, it is important to consider a ``non-diagonal'' variant of this notion with two parameters, $\eps$ and $\delta$: in this case, the condition $\mu(X_0)<\eps$ is replaced by the condition $\mu(X_0)<\delta$. In~\cite{VL23}, such entropy is referred to as mm-entropy.

The property of a measurable semimetric to be admissible can be easily described in terms of its epsilon entropy.
\begin{lemma}
A measurable semimetric is admissible if and only if its $\eps$-entropy is finite for any $\eps>0$.
\end{lemma}
Using this simple description, it is easy to understand that the sum of two admissible semimetrics is again an admissible semimetric. Indeed, the epsilon entropy of the sum of semimetrics can be estimated as follows:
$$
\heps_{2\eps}(X,\mu,\rho_1+\rho_2) \leq \heps_\eps(X,\mu,\rho_1) + \heps_\eps(X,\mu,\rho_2).
$$
Consequently, the set of all admissible semimetrics on the space $(X,\mu)$ and also the set $\Adm(X,\mu)$ are \emph{convex cones}.

We can see from the definition that for a fixed metric triple $(X,\mu, \rho)$, the function $\eps \mapsto \heps_{\eps}(X,\mu, \rho)$ is non-increasing, piecewise constant, and left-continuous. Let $\heps_{\eps+}(X,\mu, \rho) = \lim_{\delta \to 0^+} \heps_{\eps+\delta}(X,\mu, \rho)$. Then the functions $\heps_{\eps+}(X,\mu, \rho)$ and $\heps_{\eps}(X,\mu, \rho)$ are lower and upper semicontinuous on the cone $\Adm(X,\mu)$ respectively. Moreover, the following estimate holds.
\begin{lemma}\label{lemma_semicontinuity}
Let $\rho_1, \rho_2 \in \Adm(X, \mu)$ such that $||\rho_1 - \rho_2||_m < \frac{\delta^2}{4}$. Then for any $\eps > 0$,
$$
\heps_{\eps + \delta}(X,\mu, \rho_1) \le \heps_{\eps}(X,\mu, \rho_2).
$$
\end{lemma}

Another way to define an entropy of a metric triple is by approximation of its measure by discrete measures in the Kantorovich metric.
\begin{definition}\label{definition_epsilon_entropy_K}
Let $(X,\mu,\rho)$ be a metric triple with a finite first moment (i.\,e., $\rho$ is summable on $(X^2,\mu^2)$), and let $\eps>0$. Define
$$
\heps_\eps^K(X,\mu,\rho) = \inf\big\{
H(\nu) \colon d_K(\mu,\nu)<\eps
\big\},
$$
where the infimum is taken over all discrete measures $\nu$ on the metric space $(X,\rho)$, $H(\nu)$ is the Shannon entropy, and~$d_K$ is the Kantorovich distance between measures on $(X,\rho)$ (see~\cite{Kan42, V04zap, V13}).
\end{definition}

The following statements provide two-sided estimates on the given definitions of $\eps$-entropy.

\begin{lemma}\label{lemma_Kepsentropy1}
Let $(X,\mu,\rho)$ be a metric triple with a finite first moment. Let $0< \delta < \eps$ be such that for any subset $A \subset X$, if $\mu(A)<\delta$, then
\begin{equation}\label{eq11}
\int_{A\times X} \rho\, d(\mu\times \mu) < \eps - \delta.
\end{equation}
Then 
$$
\exp\Big(\heps_\eps^K(X,\mu,\rho)\Big) \leq \exp\Big(\heps_\delta (X,\mu,\rho)\Big)  + 1.
$$
\end{lemma}

\begin{lemma}\label{lemma_Kepsentropy2}
Let $(X,\mu,\rho)$ be a metric triple with a finite first moment. Then for any $\eps>0$, the following estimate holds
\begin{equation}\label{eq10}
\heps_{2\eps}(X,\mu,\rho) \leq \frac{1}{\eps} \big(\heps^K_{\eps^2}(X,\mu,\rho)+1\big).
\end{equation}
\end{lemma}

We present the proofs of Lemmas~\ref{lemma_Kepsentropy1} and \ref{lemma_Kepsentropy2} in Appendix~\ref{App_Kepsentropy_proofs}.

The following theorem from~\cite{VPZ13a} provides several equivalent reformulations of admissibility.
\begin{theorem}[Equivalent Conditions for Admissibility]
\label{Th_adm_cr}
Let $\rho$ be a measurable semimetric on $(X,\mu)$. The following statements are equivalent:
\begin{enumerate}
    \item The semimetric $\rho$ is admissible.
    \item For any $\eps>0$, the epsilon entropy $\heps_\eps(X,\mu,\rho)$ is finite.
    \item The measure $\mu$ can be approximated by discrete measures in the Kantorovich metric $d_K$; in other words, for any $\eps>0$, the epsilon entropy $\heps^K_\eps(X,\mu,\rho)$ is finite.
    \item For $\mu$-almost every $x \in X$ and for any $\eps>0$, the ball of radius $\eps$ in the semimetric $\rho$ centered at $x$ has positive measure.
    \item For any subset $A \subset X$ of positive measure, the essential infimum of the function $\rho$ on $A\times A$ is zero.
\end{enumerate}
\end{theorem}

The equivalence of the first three statements of the theorem has already been discussed above. The fifth statement of the theorem is convenient to use as a criterion for verifying the non-admissibility of a measurable semimetric.

We present another characterization of the admissibility of a semimetric from~\cite{VPZ13a}. It is given in terms of pairwise distances between a random sequence of points.

\begin{theorem}\label{Matr_dop}
Let $\rho$ be a measurable semimetric on $(X,\mu)$. Let $(x_n)_{n=1}^\infty$ be a random sequence of points chosen independently with respect to the measure $\mu$.   
\begin{enumerate} 
        \item If the metric~$\rho$ is admissible, then for any positive constant $\eps$, the probability of the following event tends to zero as $n$ goes to infinity: 
        \begin{multline}\label{eq_matrixdist}
            \text{there exists an index set $I \subset \{1, 2, \ldots, n\}$ of size}\\
            \text{at least $\eps n$, such that $\rho(x_i, x_j) > \eps$ for any distinct $i,j \in I$.}
        \end{multline}

        \item If the metric $\rho$ is not admissible, then there exists a positive constant $\eps$ such that the probability of the event~\eqref{eq_matrixdist} tends to one. 
    \end{enumerate}
\end{theorem}

In conclusion of this section, we present another theorem that allows estimating the epsilon entropy of a metric triple through the epsilon entropies of its random finite subspaces.
\begin{theorem}\label{th_eps_ent_diskr}
Let $(X, \mu, \rho)$ be a metric triple, and let $\{x_k\}_{k=1}^\infty$ be a sequence that is random with respect to the measure $\mu^\infty$. 
Let $(X_n, \mu_n, \rho_n)$ be a (random) finite metric triple, where $X_n = \{x_1, \ldots, x_n\}$, $\mu_n$ is the uniform measure on $X_n$, and $\rho_n(x_i, x_j) = \rho(x_i, x_j)$.
Then, almost surely:
\begin{enumerate}
    \item The lower estimate for $\eps$-entropy of $\rho$ holds:
    $$
    \limsup\limits_n\mathbb{H}_{\eps}(X_n, \mu_n, \rho_n) \leq \mathbb{H}_{\eps}(X, \mu, \rho);
    $$
    \item The upper estimate for $\eps$-entropy of $\rho$ holds:
    $$
    \liminf\limits_n\mathbb{H}_{\eps}(X_n, \mu_n, \rho_n) \geq \mathbb{H}_{\eps+}(X, \mu, \rho).
    $$
\end{enumerate}
\end{theorem}
It should be noted that Theorem~\ref{th_eps_ent_diskr} essentially provides an estimate of the epsilon entropy of a metric triple in terms of its matrix distribution, see Section~\ref{Sec231}.

\subsection{Convergence in the cone of admissible semimetrics}\label{sec_convergenceinmnorm}

In this section, we present a series of results from~\cite{VPZ13a}, describing the properties of the space $\MM(X,\mu)$ and the cone of admissible semimetrics $\Adm(X,\mu)$, equipped with the m-norm and the norm from $L^1(X^2,\mu^2)$.

\begin{lemma}
The space $\MM(X,\mu)$ is complete in the m-norm.
\end{lemma}

\begin{lemma}
Let a sequence of summable semimetrics $\rho_n$ on $(X,\mu)$ converges to a function $\rho$ in the m-norm. If for each $\eps>0$, for sufficiently large $n$, the estimate $\heps_\eps(X,\mu,\rho_n)<+\infty$ holds, then the function $\rho$ is an admissible semimetric.
\end{lemma}

\begin{corollary}
The limit of a sequence of admissible semimetrics in the m-norm is an admissible semimetric. The cone $\Adm(X,\mu)$ of admissible semimetrics is closed and complete in the m-norm.
\end{corollary}

The following lemma states that the limit in $L^1(X^2,\mu^2)$ of a sequence of admissible semimetrics with uniformly bounded $\eps$-entropies is an admissible semimetric.
\begin{lemma}
Let $M \subset \Adm(X,\mu)$ be such that for each $\eps>0$, the set $\{\heps_\eps(X,\mu,\rho)\colon \rho \in M\}$ is bounded. Then the closure of the set $M$ with respect to the norm of the space $L^1(X^2,\mu^2)$ lies within the cone $\Adm(X,\mu)$.
\end{lemma}

\begin{theorem}
Let a sequence of summable semimetrics $\rho_n$ converges to an admissible semimetric $\rho$ in $L^1(X^2,\mu^2)$. Then $\rho_n$ converges to $\rho$ in the m-norm.
\end{theorem}

\begin{corollary}\label{Cor_top_coinc}
On the cone $\Adm(X,\mu)$, the topology induced by the m-norm coincides with the topology induced by the standard norm from $L^1(X^2,\mu^2)$.
\end{corollary}

\subsection{Compactness and precompactness in the cone of admissible semimetrics}

According to Corollary~\ref{Cor_top_coinc}, a set of summable admissible semimetrics is compact in the m-norm if and only if it is compact in $L^1(X^2,\mu^2)$. The theorems from~\cite{VPZ13a} presented in this section provide a criterion for the precompactness of the family of admissible semimetrics in the m-norm.
\begin{theorem}\label{lemmcomp}
A set $M \subset \Adm(X,\mu)$ is precompact in the m-norm if and only if\textup:
\begin{enumerate}
    \item \textup(uniform integrability\textup) the set $M$ is uniformly integrable on $(X^2,\mu^2)$\textup;
    \item \textup(uniform admissibility\textup) for any $\eps>0$, there exist $k \geq 0$ and a partition of $X$ into sets $X_j$, $j=1,\dots,k$, such that for each semimetric $\rho \in M$, there exists a set $A \subset X$ with $\mu(A)<\eps$ and $\diam_{\rho}(X_j\setminus A) < \eps$ for all $j=1,\dots,k$.
\end{enumerate}
\end{theorem}

It is worth noting that, according to the Dunford--Pettis theorem, uniform integrability is a criterion for precompactness in the weak topology in $L^1$.

\begin{corollary}\label{cor1}
If $M\subset \Adm(X,\mu)$ is precompact in the m-norm, then its closure in $L^1(X^2,\mu^2)$ coincides with the closure in the m-norm and belongs to $\Adm(X,\mu)$. Moreover, for any $\eps>0$, the supremum $\sup\{\heps_\eps(X,\mu,\rho)\colon \rho \in M\}$ is finite.
\end{corollary}

It turns out that for convex sets consisting of admissible semimetrics, the last statement from Corollary~\ref{cor1} is also sufficient for precompactness in the m-norm.
\begin{theorem}
A convex set $M \subset \Adm(X,\mu)$ is precompact in the m-norm if and only if the supremum $\sup\{\heps_\eps(X,\mu,\rho)\colon \rho \in M\}$ is finite for any $\eps>0$.
\end{theorem}

The mentioned criteria for precompactness have an important application in dynamics. They are used to prove the criterion on the discreteness of the spectrum of a measure-preserving transformation, see Theorem~\ref{th_diskr_spectr}.

\section{Metric triples: classification, matrix distribution, and distance between triples}
\label{SecClass}

So far, when speaking about a metric triple, we had in mind a metric on a fixed standard probability space $(X,\mu)$. In this section, we temporarily deviate from this paradigm and discuss metric triples without fixing a specific standard probability space with a continuous measure.

Two summable metric triples $(X_1,\mu_1,\rho_1)$ and $(X_2,\mu_2,\rho_2)$ are isomorphic if there exists an isomorphism of measure spaces $(X_1,\mu_1)$ and $(X_2,\mu_2)$ that maps the metric $\rho_1$ to the metric $\rho_2$ (mod 0). In this section, we will also discuss the classification of admissible metric triples up to isomorphisms. Returning to a fixed space $(X,\mu)$ and the cone $\Admm(X,\mu)$ of admissible metrics on it, we are essentially talking about orbits under the action of the group of automorphisms $\aut(X,\mu)$. As mentioned in the first chapter, it was proven independently by M. Gromov (see~\cite{Gro}) and A.~Vershik (see~\cite{V98}) that the classification problem of metric triples is ``smooth'' in the sense that there exists a complete system of invariants characterizing an admissible metric triple up to isomorphisms, the matrix distribution. This complete system of invariants turned out to be not only simple and intuitive but also became a significant tool for studying metric spaces.

In Subsection~\ref{S_triples_dist} we present two  natural ways to quantitatively measure the distinction between non-isomorphic metric triples: we introduce two distances and discuss their properties.

\subsection{Classification of metric spaces with measure and matrix distribution}
\label{Sec231}

In this section, we provide a detailed exposition of what was briefly mentioned in Section~\ref{sec113} of Chapter~\ref{sec1}.

For a fixed metric triple $(X,\mu,\rho)$ and a natural number $n$, let us choose randomly and independently $n$ points $x_1,\dots, x_n$ in the space $X$ according to the distribution $\mu$. Consider the distance matrix between these points, $(\rho(x_i,x_j))_{i,j=1}^n$, which is a random distance matrix on $n$ points. Let $\mdr_n = \mdr_n(X,\mu,\rho)$ denote the resulting distribution on the space of $n\times n$ square matrices. The measure $\mdr_n$ is the image of the measure $\mu^n$ (defined on $X^n$) under the mapping $F_n\colon X^n \to M_n$,
 $$
 F_n \colon (x_1,\dots,x_n) \mapsto (\rho(x_i,x_j))_{i,j=1}^n.
 $$
\begin{definition}
 The measure $\mdr_n$ is called the \emph{matrix distribution of dimension $n$} of the metric triple $(X,\mu,\rho)$.
\end{definition} 

The following properties of finite-dimensional matrix distributions follow directly from the definition.
\begin{remark}
For a fixed metric triple $(X,\mu,\rho)$, the distributions $\mdr_n$ satisfy the following properties:
\begin{enumerate}
    \item The distribution $\mdr_n$ is concentrated on the space of distance matrices of size $n\times n$, denoted by $R_n$.
    \item The distribution $\mdr_n$ is invariant under the action of the symmetric group $S_n$ by simultaneous permutations of rows and columns.
    \item The distribution $\mdr_n$ is the projection of the distribution $\mdr_{n+1}$ onto $n\times n$ matrices formed by the first $n$ rows and columns.
\end{enumerate}
\end{remark}

The symbol $\mdrinf = \mdrinf(X,\mu,\rho)$ denotes the projective limit of finite-dimensional distributions $\mdr_n$, a distribution on the space $M_\infty$ of $\mathbb{N}\times\mathbb{N}$ matrices. This measure is concentrated on the space of infinite distance matrices (symmetric matrices with non-negative real coefficients satisfying the triangle inequality), denoted as $R_\infty$. Moreover, the measure $\mdrinf$ is invariant under the action of the group $S^\infty$, which simultaneously permutes the columns and rows of matrices (briefly, the infinite diagonal symmetric group). The measure $\mdrinf$ is the image of the Bernoulli measure $\mu^\infty$ (defined on $X^\infty$) under the mapping $F_\infty\colon X^\infty \to M_\infty$,
\begin{equation}\label{eq010301}
 F_\infty \colon (x_j)_{j=1}^\infty \mapsto (\rho(x_i,x_j))_{i,j=1}^\infty.
\end{equation}
A random matrix with respect to the distribution $\mdrinf$ is the distance matrix $(\rho(x_i,x_j))_{i,j=1}^\infty$, where the sequence of points $(x_i)_{i=1}^\infty$ in the space $X$ is chosen randomly and independently according to the distribution $\mu$.

\begin{definition}[see~\cite{V98}]
The measure $\mdrinf$ is called the \emph{matrix distribution of the metric triple} $(X,\mu,\rho)$.
\end{definition}

A theorem stating that the matrix distribution completely determines a metric triple up to isomorphism was established by M.~Gromov and A.~Vershik independently.

M.~Gromov: The set of measures $\mdr_n$, $n\in \mathbb{N},$ forms a complete system of invariants for the metric triple $(X,\mu,\rho)$. That is, the necessary and sufficient condition for the equivalence of two triples is the coincidence of the corresponding measures $\mdr_n$ for all $n$.

A.~Vershik: The measure $\mdrinf$ on the space of infinite distance matrices is a complete invariant for the metric triple $(X,\mu,\rho)$. In other words, two admissible non-degenerate metrics are isomorphic if and only if their matrix distributions coincide.

The equivalence of the two conclusions is obvious. The first statement was proved by M.~Gromov and the proof is based on rather specialized analytical considerations. After the theorem was communicated to A.~Vershik, he provided an entirely different proof based on a simple ergodic theorem (Borel's strong law of large numbers for a sequence of independent random variables). The analysis of both proofs is given in the book~\cite{Gro}. We will present the second proof. Both proofs were obtained in the late 1990s and were published: the first in~\cite{Gro}, the second in~\cite{V02mccme}.

\begin{theorem}[M.~Gromov, A.~Vershik]\label{th_iso}
Two metric triples $(X_1,\mu_1,\rho_1)$ and $(X_2,\mu_2,\rho_2)$ are isomorphic if and only if the matrix distributions $\mdrinf(X_1,\mu_1,\rho_1) $ and  $\mdrinf(X_2,\mu_2,\rho_2) $ of these triples coincide.
\end{theorem}
\begin{proof}
The necessity of the condition is evident. To prove sufficiency, we will show that the metric triple $(X,\mu,\rho)$ is uniquely determined by its matrix distribution.

Recall that we can assume that the metric space $(X,\rho)$ is complete, and the measure $\mu$ is non-degenerate, i.\,e., its support coincides with the entire space (there are no non-empty open sets of measure zero). From this and the ergodic theorem, it follows that almost any sequence of points with respect to $\mu^{\infty}$ is dense in $(X,\rho)$. Consequently, the space that is the closure (more precisely, the completion) of almost any sequence is the entire $X$. It remains to reconstruct the measure on the space. Using the same ergodic theorem, the measure of any ball, and even the intersection of a finite set of balls centered at points of our sequence, can be uniquely reconstructed as the density of points of this sequence lying in such an intersection. As it is known, the algebra of sets spanned by the set of all (or almost all) balls of arbitrary radius in a separable metric space is dense in the algebra of all measurable sets. Therefore, the value of the measure on this algebra uniquely determines the measure on the entire space $X$. Thus, we have reconstructed the metric triple from its matrix distribution.
\end{proof}

\begin{remark}
The conclusion of Theorem~\ref{th_iso} fails if we replace admissible metrics with admissible semimetrics. A counterexample can be provided by an admissible semimetric that does not distinguish pairs of points and a metric obtained from it by a factorization that identifies points with zero distance.

Also, the conclusion of Theorem~\ref{th_iso} fails if we abandon the condition of metric admissibility. Metrics obtained from the aforementioned semimetrics by adding a constant have equal matrix distributions but are not isomorphic.
\end{remark}

\subsection{Characterization of matrix distributions}
\label{Sec232}
The Gromov--Vershik Theorem~\ref{th_iso} exhaustively classifies metric triples in terms of matrix distributions, i.\,e., probability measures on the set $R_\infty$ of infinite distance matrices. However, a question that is essential for the final resolution of classification problems and remains open is to describe possible values of the invariants, in this case, to describe {\it which measures can be matrix distributions of mm-spaces}. It was noted earlier that the distributions must be invariant and ergodic with respect to the action of the diagonal symmetric group. But this condition is not sufficient. It should be mentioned that the set of all invariant measures on the set of infinite symmetric matrices (not necessarily distance matrices) with respect to the diagonal group was found by D.~Aldous (\cite{Au}), and its intersection with the ergodic measures on $R_\infty$ is much broader than the set of matrix distributions. The problem we consider consists of the classification of metrics as measurable functions of two variables with respect to the simultaneous action on both variables by a group of transformations that preserve the measure. Under certain conditions, this is precisely the property that distinguishes matrix distributions from a wider class of measures on infinite matrices.

Let us note the following obvious property of metrics $\rho(\cdot,\cdot)$ as measurable functions of two variables. A measurable symmetric function of two variables is called {\it pure} if the mapping $x \mapsto \rho(x, \cdot)$ is injective mod 0, i.\,e., for almost all pairs $x \neq x'$, the corresponding one-variable functions $f(x,\cdot)$ and $f(x',\cdot)$ are not almost everywhere equal. It is evident that a metric is a pure function (a semimetric is not necessarily). It is not difficult to prove the following property.
\begin{lemma} 
The following two properties of a metric triple $(X,\mu,\rho)$ are equivalent:
\begin{enumerate}
    \item The group of $\mu$-preserving isometries mod 0 on the space $(X,\rho)$ is trivial (consists of the identity transformation). In this case, we will say that the admissible metric is \emph{irreducible}.
    \item The map $F_\infty$ from $(X^\infty,\mu^\infty)$ to $M_\infty$ is an isomorphism onto its image, see~\eqref{eq010301}.
\end{enumerate}
\end{lemma}
Thus, the \emph{matrix distribution of an irreducible metric} is an isomorphic image of the Bernoulli measure.
Thus, our task is to describe isomorphic images of
Bernoulli measures under the mapping $F_\infty$. These images, as measures on distance matrices, will be referred to as \emph{simple measures}. The internal description of simple measures is associated with a more detailed consideration of sigma-subalgebras on which the measures are defined, and we will not dwell on this here. In~\cite{XV}, such a description is given for a similar case, namely for not necessarily symmetric functions of several variables. On the other hand, in~\cite{V22}, for measures on enumerations of partially ordered sets, a concept equivalent to the concept of measure dimension on matrices is introduced; simplicity corresponds to dimension one. In connection with Aldous's theorem, it is relevant to mention that the absence of the notion of dimension (simplicity) makes it difficult to understand why invariant measures in this theorem split into two different classes; the distinction is in which sigma-subalgebras the measures are defined on. A detailed overview of the relevant concepts and their applications will be addressed in another paper.

In conclusion of the section, we present a theorem that provides an entropy-based description of matrix distributions of metric triples.

Let $\eps>0$. For a finite distance matrix of size $n\times n$, we define its $\eps$-entropy as the $\eps$-entropy of a $n$-point space with the uniform measure and the metric defined by this matrix. We will say that an infinite distance matrix $A=(a_{i,j})_{i,j=1}^\infty$ is \emph{entropy admissible} if for each $\eps>0$ the $\eps$-entropies of its corner minors $n\times n$, i.\,e. matrices $(a_{i,j})_{i,j=1}^n$, are uniformly bounded with respect to $n$. We say that a matrix $A$ is \emph{summable} if there exists a finite limit $\lim\limits_{n\to \infty} \frac{1}{n^2}\sum\limits_{i=1}^n\sum\limits_{j=1}^n a_{i,j}$.

\begin{theorem}[Entropy characteristic of matrix distributions of admissible metrics]
The matrix distribution $\mdrinf(X,\mu,\rho)$ of a summable metric triple $(X,\mu,\rho)$ is a measure on $R_\infty$ that is ergodic with respect to simultaneous permutations of rows and columns and is concentrated on summable entropy admissible matrices. Conversely, any ergodic measure $D$ that is concentrated on summable entropy admissible matrices is the matrix distribution of some summable metric triple $(X,\mu,\rho)$.
\end{theorem}

We comment on the proof of the theorem. The summability and entropy admissibility of $\mdrinf$-almost every matrix follows from the law of large numbers and Theorem~\ref{th_eps_ent_diskr}. Conversely, let $A=(a_{i,j})_{i,j=1}^\infty$ be a summable entropy admissible distance matrix (almost every point from the support of measure $D$). For each $n$, let $(X_n,\mu_n,\rho_n)$ be a $n$-points semimetric space defined by the matrix $(a_{i,j})_{i,j=1}^n$ and equipped with a uniform measure. The summability and entropy admissibility conditions of the matrix $A$ allow us to prove that the sequence of semimetric triples $(X_n,\mu_n,\rho_n)$ is precompact with respect to the special metric $\mdist$, see Definition~\ref{def_mdist} and Theorem~\ref{th_two_metrics2} below. Let an admissible semimetric triple $(X,\mu,\rho)$ be a limit point of this sequence of triples. The matrix distribution $\mdrinf(X,\mu,\rho)$ is the weak limit of a subsequence of matrix distributions $\mdrinf(X_n,\mu_n,\rho_n)$, thus it turns out to be equal to the original measure $D$ due to its ergodicity.

\subsection{Two metrics on metric triples}
\label{S_triples_dist}

Let $(X_1,\mu_1,\rho_1)$ and $(X_2,\mu_2,\rho_2)$ be two semimetric triples. We consider two ways to measure the distance between them, both in the spirit of the Gromov--Hausdorff distance. The first one is to embed both measure spaces into a space $(X,\mu)$ and minimize the distance between the semimetrics in the $m$-norm.
\begin{definition}\label{def_mdist}
The distance $\mdist$ between two semimetric triples $(X_1,\mu_1,\rho_1)$ and $(X_2,\mu_2,\rho_2)$ is defined as the infimum over all possible couplings $(X,\mu)$ of the measure spaces $(X_1,\mu_1)$ and $(X_2,\mu_2)$ (with projections $\psi_1 \colon X \to X_1$, $\psi_2 \colon X \to X_2$) of the $m$-distances between the semimetrics $\rho_1\circ \psi_1$ and $\rho_2\circ \psi_2$ on the space $(X,\mu)$:
\begin{align*}
\mdist\Big((X_1,\mu_1,\rho_1),&(X_2,\mu_2,\rho_2)\Big) 
=\\
&=\inf\Big\{\|\rho_1\circ \psi_1 - \rho_2\circ \psi_2\|_m \colon \psi_{1,2} \colon (X,\mu) \to (X_{1,2}, \mu_{1,2})\Big\}.
\end{align*}
\end{definition}

The second natural way to measure the distance between semimetric triples is to isometrically embed both semimetric spaces into a semimetric space $(X,\rho)$ and minimize the distance between the measures, for instance, using the Kantorovich metric $d_K$.
\begin{definition}
The distance $\kdist$ between two semimetric triples $(X_1,\mu_1,\rho_1)$ and $(X_2,\mu_2,\rho_2)$ is defined as the infimum over all possible isometric embeddings $\phi_1\colon (X_1,\rho_1) \to (X,\rho)$ and $\phi_2\colon (X_2,\rho_2) \to (X,\rho)$ of the Kantorovich distances on the space $(X,\rho)$ between the measures $\phi_1(\mu_1)$ and $\phi_2(\mu_2)$:
\begin{align*}
\kdist\Big((X_1,\mu_1,\rho_1),&(X_2,\mu_2,\rho_2)\Big) =\\
=&\inf\Big\{d_K(\phi_1(\mu_1),\phi_2(\mu_2))\colon \phi_{1,2}\colon (X_{1,2},\rho_{1,2}) \to (X,\rho)\Big\}.
\end{align*}
\end{definition}

The following theorem states that the two distances described above are equivalent.
\begin{theorem}\label{th_two_metrics2}
For any integrable semimetric triples $(X_1,\mu_1,\rho_1)$ and $(X_2,\mu_2,\rho_2)$, the following inequalities hold:
\begin{align*}
\kdist\Big((X_1,\mu_1,\rho_1),(X_2,\mu_2,\rho_2)\Big) &\leq \mdist\Big((X_1,\mu_1,\rho_1),(X_2,\mu_2,\rho_2)\Big),\\
\mdist\Big((X_1,\mu_1,\rho_1),(X_2,\mu_2,\rho_2)\Big)
&\leq 2 \kdist\Big((X_1,\mu_1,\rho_1),(X_2,\mu_2,\rho_2)\Big).
\end{align*}
\end{theorem}

The proof of Theorem~\ref{th_two_metrics2} (as well as the subsequent Theorem~\ref{th_comp_triples}) is provided in Appendix~\ref{App_th_triples}.

\begin{remark}
The functions $\kdist$ and $\mdist$ are metrics on the set of equivalence classes of admissible integrable metric triples. In particular, they are equal to zero if and only if two admissible metric triples are isomorphic. The convergence of a sequence of admissible metric triples in either of these metrics implies weak convergence of their matrix distributions.

The set of all admissible semimetric triples is complete with respect to $\kdist$ and $\mdist$ distances.

The functions $\kdist$ and $\mdist$ are semimetrics on the set of equivalence classes of admissible integrable semimetric triples. For instance, the distance between an admissible semimetric triple and a metric triple obtained from it by factorization over the sets of zero diameters is zero.
\end{remark}

The following theorem is an analog of Theorem~\ref{lemmcomp} and provides a criterion for the precompactness of a family of admissible semimetric triples in the $\mdist$-metric (and $\kdist$-metric).

\begin{theorem}\label{th_comp_triples}
A set $M = \{(X_i,\mu_i,\rho_i) : i \in I\}$ of summable admissible semimetric triples is precompact in the $\mdist$-metric (and $\kdist$-metric) if and only if\textup:
\begin{enumerate}
    \item the semimetrics $\rho_i$ are uniformly integrable:
    $$
    \lim_{R \to \infty} \sup_{i\in I} \int_{\rho_i > R} \rho_i \, d\mu_i^2  = 0;
    $$
    
    \item for any $\eps>0$, the epsilon entropies are uniformly bounded:
    $$
    \sup_{i \in I} \heps_\eps(X_i,\mu_i, \rho_i) < \infty. 
    $$
\end{enumerate}
\end{theorem}

A slightly different version of Theorem~\ref{th_comp_triples} can be found in~\cite{GPW}, where various distances between metric triples and their properties are considered; see also~\cite{St}.

In~\cite{GadKrish}, the authors investigate Lipschitz properties of mappings that associate the finite-dimensional matrix distribution $\mdr_n$, $n\geq 1$, to a metric triple. To do this, they introduce a specific distance on the space of square matrices $M_n$ of dimension $n$, and the distance between distributions is defined using the Prokhorov metric, which metrizes the weak topology. Let us formulate a result of a similar kind for infinite-dimensional matrix distributions $\mdrinf$.

It should be noted that if $\rho$ is a summable semimetric on the space $(X,\mu)$, then the random distance matrix $M$ is almost surely summable, meaning it has a finite mean:
$$
\av(M) := \lim_{n \to \infty} \frac{1}{n^2} \sum_{i=1}^n\sum_{j=1}^n M_{i,j} = \int_{X^2}\rho\, d\mu^2.
$$
For two summable distance matrices $A$ and $B$, we define a distance between them as follows:
$$
\MatrixDist(A,B) = \inf\big\{\av(D)\colon |A_{i,j}-B_{i,j}|\leq D_{i,j},\, \forall \, i,j \in \mathbb{N}\big\}, 
$$
where the infimum is taken over all possible summable distance matrices $D$ that element-wise majorize the difference between matrices $A$ and $B$. We note the similarity between the definition of the semimetric $\MatrixDist$ and the $m$-norm (see Definition~\ref{def_mnorm}).

Using the distance $\MatrixDist$ on the space of summable distance matrices, we can use the corresponding Kantorovich distance between matrix distributions.
\begin{theorem}
The Kantorovich distance between matrix distributions of two metric triples does not exceed the distance $\mdist$ between these triples.
\end{theorem}

        \subsection{The Urysohn universal metric space}

    An outstanding discovery made by P.~S.~Urysohn (1898--1924) and published in his posthumous paper of 1927 \cite{U1} was the construction of a universal complete separable metric space. This space is now referred to as Urysohn's space. It is the unique up to isometry complete separable metric space $(\Ury,\rho_U)$ possessing two properties: universality, meaning that any separable metric space can be isometrically embedded into $\Ury$, and homogeneity, which means that for any two isometric compact subsets of the space $\Ury$ and any isometry between them, there exists an extension of this isometry to an isometry of the entire space onto itself. After being largely forgotten for many years, starting from the 2000s this space has become a subject of study for many mathematicians. We highlight one of the many results. One of the significant properties of Urysohn's space is its ``typicality'' in the following sense.

\begin{theorem}[\cite{V04}]\label{th_ury}
Consider the set $\mathrm{M}$ of all metrics on a countable set {\emph{N}} (for example, on the set of natural numbers) and equip it with the natural weak topology. Then, an everywhere dense $G_{\delta}$ subset (i.\,e., ``typical'') in $\mathrm{M}$ consists of metrics for which the completion of the set {\emph{N}} with respect to this metric is isometric to the Urysohn space $(\Ury,\rho_U)$.
\end{theorem}

The proof of this fact refines the construction of the Urysohn space: the space is constructed using an inductive process that defines a distance matrix (Urysohn matrix), see~\cite{V04}. Surely, by defining a probabilistic Borel non-degenerate continuous measure $\mu$ on the Urysohn space, we obtain a metric triple $(\Ury,\mu,\rho_U)$. This raises a question: how large is the part of the space of all metric triples $\Admm(X,\mu)$ on a Lebesgue space $(X,\mu)$ that consists of spaces that are isomorphic to the Urysohn space?

Consider the weak topology on the space of metric triples. To do this, we will identify each metric triple with its matrix distribution (a probability measure on the space $R_\infty$ of distance matrices). The topology on the matrix distributions is the weak topology on the space of measures on $R_\infty$, and it induces the weak topology on the triples. According to the aforementioned Theorem~\ref{th_ury}, the set of \emph{Urysohn} matrices (defining a metric space whose completion is the Urysohn space) is a dense $G_\delta$ set in $R_\infty$. Hence, the set of probability measures concentrated on Urysohn matrices is a dense $G_\delta$ subset of the space of probability measures on $R_\infty$. It seems that the set of matrix distributions concentrated on Urysohn matrices is a dense $G_\delta$ subset of the set of matrix distributions of metric triples.

A question of what type of metric spaces forms a typical set in the space of triples with respect to the norm topology is open.

A natural and highly important question is how to define non-degenerate continuous measures on the Urysohn space. Are there any distinguished measures among them, like the Wiener measure in the space $C(0,1)$ of continuous functions on $[0,1]$ (notably, this space is universal but not homogeneous)?
This question is open and is related to another significant question: whether there exists any distinguished structure in the space $\Ury$. Recall (see~\cite{CV2005}) that a structure of an abelian continuous (not locally compact) group can be introduced in $\Ury$, although, unfortunately, not in a unique way.
Each non-degenerate measure on $\Ury$ corresponds to a matrix distribution, which allows us to pose the following question: to find matrix distributions concentrated on Urysohn matrices.
In fact, this problem reduces to constructing an ergodic measure concentrated on the set of Urysohn distance matrices and invariant under the infinite symmetric group.

\chapter{Dynamics on admissible metrics}

\label{sec3}

\section{Scaling entropy}

We assume the reader to be familiar with the basics of the classical entropy theory (see, e.\,g., papers~\cite{Kolm58,KSF,OW87,Ro67}). As we mentioned above, we study a new variant of entropy theory based on the dynamics of admissible metrics. In fact, the usage of a metric was mentioned in the paper by Shannon but was never developed further and has been forgotten. The introduction of a metric helps to determine new properties of an automorphism and allows one to extend Kolmogorov's theory to zero entropy automorphisms. Similar ideas were proposed in papers~\cite{F,KT}, see also a survey~\cite{KanKWei}. Following the definition by A. Vershik, we formulate the theory of scaling entropy, the beginnings of which were stated in papers~\cite{V10a, V10b, V11}. 

\subsection{Definition of scaling entropy. Asymptotic classes}

\label{sec311}
Let $T$ be an automorphism of a standard probability space $(X,\mu)$. Recall that for a summable admissible semimetric $\rho \in \Adm(X,\mu)$ we denote by $T^{-1}\rho$ its translation by the automorphism $T$ and by $T_{av}^n\rho$ the average of its first $n$ translations:
$$ T_{av}^n\rho(x,y) = \frac{1}{n} \sum\limits_{i =0}^{n-1} \rho(T^i x, T^i y). $$
In this section, we study the asymptotic behavior of the sequence of epsilon-entropies of semimetric triples $(X,\mu,T_{av}^n\rho)$. Recall that the epsilon-entropy $\heps_\eps(X,\mu,\rho)$ of a semimetric triple is the logarithm of the minimal possible number of balls of radius $\eps$ that cover the whole space~$X$ up a set of measure less than~$\eps$ (see Definition~\ref{definition_epsilon_entropy}). Since we are only interested in the asymptotics of such functions we will consider them up to asymptotic equivalence in the following sense.

For two sequences $h = \{h_n\}$ and $h' = \{h'_n\}$ of nonnegative numbers we write $h \preceq h^\prime$ if $h_n=O(h^\prime_n)$ and we write $h_n \asymp h^\prime_n$ if both relations $h\preceq h^\prime$ and $h^\prime \preceq h$ hold, in this case we say that the sequences~$h$ and~$h^\prime$ are \emph{equivalent}. 

\begin{definition}\label{definition_asymp_class}
    For two functions $\Phi,\Psi \colon \mathbb{R}_+\times \mathbb{N} \to \mathbb{R}_+$ we write $\Phi \preceq \Psi$ if for any $\eps>0$ there exists $\delta>0$ such that
    $$
    \Phi(\eps, n) \preceq \Psi(\delta, n), \qquad n \to \infty. 
    $$
    In this case, we say that $\Psi$ \emph{asymptotically dominates}~$\Phi$
    
    We call $\Phi$ and $\Psi$ \emph{equivalent} and write $\Phi \asymp \Psi$ if $\Psi \preceq \Phi \preceq \Psi$. We denote by~$[\Phi]$ the equivalence class of a given function $\Phi$ with respect to relation~$\asymp$ and call it \emph{the asymptotic class} of~$\Phi$. Relation~$\preceq$ extends naturally to equivalence classes and forms a partial order on the set of asymptotic classes. 
\end{definition}

Let us note that the equivalence class of a function~$\Phi$ is in fact defined in two steps: first, for a fixed~$\eps$  we consider the family of all sequences equivalent to $\Phi(\eps, n)$ when $n$ goes to infinity and then we identify all functions $\Phi$ with a given least upper bound of such families when $\eps$ goes to zero (see Section~\ref{section_scentropy_values}). Hence  the class~$[\Phi]$ can be viewed as a germ of the asymptotics of $\Phi$ at~$(0, +\infty)$.

For a given semimetric $\rho \in \Adm(X,\mu)$ we consider a function $\Phi_\rho \colon \mathbb{R}_+\times \mathbb{N} \to \mathbb{R}_+$ defined as follows 
\begin{equation}\label{Phi_metr}
\Phi_\rho(\eps, n) = \heps_\eps(X,\mu, T_{av}^n\rho).
\end{equation}

The following theorem states that the asymptotic class of the function $\Phi_\rho$ does not depend on the choice  of admissible metric~$\rho$, or, more generally, generating semimetric. A semimetric $\rho$ on $(X,\mu)$ is called \emph{generating} for~$T$ (or $T$-generating) if its translations by the action of $T$ separate points mod 0 that is for some subset $X_0\subset X$ of full measure for any $x,y \in X_0$ there is some~$n$ such that $T^n\rho(x,y)>0$. Clearly, any metric forms a generating semimetric. Note that the definition of a generating semimetric is a direct analog of a generating partition in classical entropy theory (see.~\cite{Ro67}). The existence of countable and finite generating partitions was studied in~\cite{Ro67, Kr}. The following theorem, proved in~\cite{Z15a}, naturally corresponds to the Kolmogorov--Sinai theorem which states that Kolmogorov entropy does not depend on a generating partition. 

\begin{theorem}\label{theorem_class_invariance}
Let $\rho_1, \rho_2 \in \Adm(X,\mu)$ be $T$-generating semimetrics. Then the classes $[\Phi_{\rho_1}]$ and $[\Phi_{\rho_2}]$ coincide.
\end{theorem}
The proof of Theorem~\ref{theorem_class_invariance} is based on the following lemma.
\begin{lemma}\label{lem_asymp_ineq}
Let $\rho_1,\rho_2 \in \Adm(X,\mu)$. If $\rho_1$ is $T$-generating then for any $\eps>0$ there exists $\delta>0$ such that
$$
\heps_\eps(X,\mu,T_{av}^n\rho_2) \preceq \heps_\delta(X,\mu,T_{av}^n\rho_1), \qquad n \to +\infty. 
$$
\end{lemma}
An outline of the proof of Lemma~\ref{lem_asymp_ineq} and its more general version which makes it possible to refine the invariant is given in Appendix~\ref{section_exp_scaling_entropy}.  

Theorem~\ref{theorem_class_invariance} allows us to give the following definition of the scaling entropy of a dynamical system~$(X,\mu,T)$. 
\begin{definition}\label{definition_scaling_entropy}
\emph{The scaling entropy} $\scfunclass(X,\mu, T)$ of a system $(X,\mu,T)$ is the asymptotic class $[\Phi_\rho]$ for some (hence for any) $T$-generating semimetric $\rho \in \Adm(X,\mu)$.
\end{definition}
Let us emphasize that the class $\scfunclass(X,\mu,T)$ is preserved under isomorphism and is \emph{a measure theoretic invariant of dynamical systems}. Let us also note the scaling entropy is a monotone function with respect to a factor map:
\begin{remark}
Assume that a system $(X_2,\mu_2,T_2)$ is a factor of another system $(X_1,\mu_1,T_1)$. Then $\scfunclass(X_2,\mu_2, T_2) \preceq \scfunclass(X_1,\mu_1, T_1)$.
\end{remark}

    \subsubsection{Example of computation of the scaling entropy: Bernoulli shift}
		
    The invariant definition of the scaling entropy using measurable metrics in many cases makes it possible to establish connections between measure-theoretic and topological dynamics (see, e.\,g., \cite{S, Vep20b, Vep22}). However, for an explicit computation of the scaling entropy of particular dynamical systems, it is often useful to choose a generating partition and the corresponding cut semimetric (which is a generating one as well). In this case, the computation of our invariant reduces to consideration of a sequence of finite-dimensional cubes each endowed with the Hamming distance and a projection of a certain stationary measure~$\mu$ onto the first $n$~coordinates. As an example of such a computation, we consider the classical Bernoulli shift on the binary alphabet. 
	
	\begin{theorem}
        The shift map on the space of all binary sequences $X = \{0,1\}^\mathbb{Z}$ with the Bernoulli measure~$\mu$ with parameter~$1/2$ has scaling entropy $\scfunclass = [n]$.
	\end{theorem}
	\begin{proof}
        Since the scaling entropy does not depend on the choice of the metric it is enough for our computation to consider only a cut semimetric~$\rho$ corresponding to the partition into preimages of the first coordinate. That is for two sequences $x,y \in \{0,1\}^\mathbb{Z}$ the distance in the semimetric~$\rho$  between~$x$ and~$y$ is $|x_0 - y_0|$.
		
        The average $T_{av}^n \rho$ is the Hamming distance corresponding to the first $n$ coordinates. Therefore, the semimetric triple $(\{0,1\}^\mathbb{Z}, \mu, T_{av}^n \rho)$ up to factorization by the sets of diameter $0$ is isomorphic to the binary cube of dimension~$n$ with the uniform measure and Hamming distance.
		
        Let us fix $\eps \in (0, \frac{1}{2})$. The measure of any ball of radius~$\eps$ does not exceed $2^{-c(\eps)n }$ due to the central limit theorem for the Bernoulli scheme. Hence,
		$$
		c(\eps) n \le \heps_\eps(X,\mu,T_{av}^n\rho) \le n,
		$$
        where the upper bound is straightforward -- the logarithm of the number of elements in the binary cube of dimension~$n$. Therefore, the asymptotic class of $\Phi_\rho(\eps, n) = \heps_\eps(X,\mu, T_{av}^n\rho)$ coincides with the class of the function $(\eps,n) \mapsto n$. Hence, $\scfunclass = [n]$.
	\end{proof}

        \subsection{Definition of scaling entropy via Kantorovich distance}
        In the definition of scaling entropy, instead of $\eps$--entropy $\heps_{\eps}(X,\mu,\rho)$ we could use the value $\heps^K_{\eps}(X,\mu,\rho)$ which is defined using an approximation in Kantorovich distance of the measure~$\mu$ by discrete measures (see Definition~\ref{definition_epsilon_entropy_K}). However, by doing so we would obtain the same invariant: Lemma~\ref{lemma_Kepsentropy1} and Lemma~\ref{lemma_Kepsentropy2} show that the asymptotic behaviors of $\heps^K_{\eps}(X,\mu, T_{av}^n\rho)$ and $\heps_{\eps}(X,\mu,T_{av}^n\rho)$ coincide. Hence, we have the following proposition. 
	
	\begin{proposition}
        Let $T$ be a measure-preserving transformation of a measure space $(X,\mu)$ and $\rho \in \Adm(X,\mu)$ be a $T$-generating semimetric. Then 
		$$
		\heps^K_{\eps}(X,\mu,T_{av}^n\rho) \in \scfunclass(X,\mu,T).
		$$
	\end{proposition}
	\begin{proof}
        Find some $\eps,\delta > 0$ satisfying the assumptions of Lemma~\ref{lemma_Kepsentropy1}. Notice that inequality~\eqref{eq11} with given $\eps$ and $\delta$ holds for all averages~$T_{av}^n\rho$ of the semimetric~$\rho$. Applying Lemmas~\ref{lemma_Kepsentropy1} and~\ref{lemma_Kepsentropy2} to semimetric~$T_{av}^n\rho$ we obtain  
        $$
		\heps^K_{\eps}(X,\mu,T_{av}^n\rho) \le 2\heps_{\delta}(X,\mu,T_{av}^n\rho) \le \frac{4}{\delta}(\heps^K_{\delta^2/4}(X,\mu,T_{av}^n\rho) + 1).
		$$
        Therefore, $\heps^K_{\eps}(X,\mu,T_{av}^n\rho) \in [\Phi_\rho] = \scfunclass(X,\mu,T)$.
	\end{proof}

\subsection{Definition of scaling entropy via partition function} 

\label{Sec313}

    Instead of the usual average of a metric over $n$ iterations of a transformation~$T$, we could consider a partition function from Section~\ref{section_catalyst} (see also~\cite{V23}) that is a weighted average with exponentially decaying weights. Recall that
    $$ \Omega_T(\rho,z)\equiv \Omega_T(z)=(1-z)\sum_{n=0}^{\infty} z^n \rho(T^n x, T^n y),\quad z\in [0,1). $$
    Then we can consider an entropy function defined similarly to formula~\eqref{Phi_metr}:
    \begin{equation*}
    \tilde\Phi_\rho(\eps, z) = \heps_\eps(x,\mu, \Omega_T(\rho,z)).
    \end{equation*}
    Following Definition~\ref{definition_asymp_class} we can consider the asymptotic class~$[\tilde\Phi_\rho]$. Then following the proof of Theorem~\ref{theorem_class_invariance} it is not difficult to show that this class does not depend on the choice of a generating semimetric $\rho \in \Adm(X, \mu)$ and is a measure-theoretic invariant. However, this invariant provides nothing essentially new: the class~$[\tilde\Phi_\rho]$ is completely determined by the scaling entropy~$\scfunclass(T)$ as follows.
    \begin{proposition}\label{proposition_statsum}
     Let $\rho \in \Adm(X, \mu)$ be a generating semimetric. Then the function $\tilde\Phi_\rho(\eps, 1 - \frac{1}{n})$  belongs to the class~$\scfunclass(T)$.
    \end{proposition}

    The proof of Proposition~\ref{proposition_statsum} is given in Appendix~\ref{section_proof_statsum}. 
    
\section{Scaling entropy sequence}

Before we begin to discuss the properties of scaling entropy in the general case, we will focus on an important particular case when the asymptotic behavior of epsilon-entropies of the averages does not essentially depend on epsilon. In this case, our invariant can be significantly simplified and it becomes a class of asymptotically equivalent sequences which we call \emph{scaling entropy sequence}. We call an automorphism that has such a sequence \emph{stable}. Bernoulli shifts as well as all transformations with positive Kolmogorov entropy, transformations with pure point spectrum, and many others (see Section~\ref{section_scaling_entropy_sequence_values}) are stable. In fact, in the pioneering papers~\cite{V10a, V10b, V11} only the case of a stable transformation was considered and it was conjectured to be the general case. As we will see in Section~\ref{section_nonstable_example} there exist ergodic automorphisms that are not stable. Despite that, the scaling entropy sequence plays an important role in the theory of scaling entropy which we present here. The case of a stable transformation was studied in~\cite{V10a, V10b, V11, VPZ13a, Z15a, Z15b, PZ15}. 

\subsection{Definition of scaling sequence. Stable classes}

\label{Sec321}
{
\begin{definition}\label{definition_stable_class}
 We call an asymptotic class $\mathcal{H}$ \emph{stable} if it contains a function  $\Phi(\eps, n)$ that does not depend on~$\eps$ that is $\Phi(\eps,n) = h_n$. A dynamical system $(X,\mu, T)$ is called \emph{stable} if the class $\scfunclass(X,\mu, T)$ is stable.
\end{definition}
}

Note that for two functions $\Phi(\eps, n) = h_n$ and $\Phi^\prime(\eps, n) = h^\prime_n$ relations~$\preceq$ and~$\asymp$ are satisfied if and only if they are satisfied for the corresponding sequences $h = \{h_n\}$ and $h^\prime = \{h^\prime_n\}$.

\begin{definition}\label{def_sc_ent}
Let be a summable generating admissible semimetric on~$(X,\mu)$. A non-decreasing sequence $h$, $h = \{h_n\}_{n \in \mathbb{N}}$, of positive numbers is called \emph{scaling} for $(X,\mu,T,\rho)$ if for any sufficiently small positive~$\eps$ the following holds
$$\heps_\eps(X,\mu,T_{av}^n\rho) \asymp h_n.
$$
It makes sense to consider the entire class of equivalent sequences. Indeed, if a sequence $h$ is a scaling sequence then any other sequence $h'$ is scaling if and only if $h' \asymp h$. We denote the class of scaling sequences by $\scclass(X,\mu,T,\rho)$.
\end{definition}

The class $\scclass(X,\mu,T,\rho)$ is a section of the asymptotic class $[\Phi_\rho]$ by the set of all sequences. Therefore, the following theorem which was proved in~\cite{Z15a} follows from the Invariance Theorem~~\ref{theorem_class_invariance}.

\begin{theorem}\label{th1}
Let $\rho_1, \rho_2 \in \Adm(X,\mu)$ be admissible $T$-generating semimetrics. Then
$\scclass(X,\mu,T,\rho_1) = \scclass(X,\mu,T,\rho_2)$.
\end{theorem}

Theorem~\ref{th1} allows us to give the following definition of the scaling entropy sequence of a dynamical system $(X,\mu,T)$.
\begin{definition}
A sequence $h =\{h_n\}$ is called \emph{a scaling sequence} of a system $(X,\mu,T)$ if $h \in \scclass(X,\mu,T,\rho)$ for some (hence for any) $T$-generating semimetric $\rho \in \Adm(X,\mu)$. We denote by $\scclass(X,\mu,T)$ the class of all scaling sequences of  $(X,\mu,T)$.
\end{definition}
Let us emphasize that the class $\scclass(X,\mu,T)$ of all scaling sequences is a measure-theoretic invariant of dynamical systems. However, unlike the class $\scfunclass$ the class of sequences $\scclass$ may be empty for certain systems as we will see in Section~\ref{section_nonstable_example}.

\subsection{Possible values of scaling sequence}\label{section_scaling_entropy_sequence_values}

One of the first naturally arising questions -- what values can the introduced invariant attend that is which sequences can be scaling sequences for some dynamical systems? The following two theorems proved in~\cite{Z15b, PZ15} give the description of all possible values of the scaling sequence.
\begin{theorem}\label{theorem_subadd_entropy_sequence}
If the class $\scclass(X,\mu,T)$ of all scaling sequences is non-empty then it contains a non-decreasing subadditive sequence of positive numbers.
\end{theorem}

\begin{theorem}\label{theorem_example_entropy_sequence}
Any subadditive non-decreasing sequence of positive numbers is a scaling sequence for some ergodic system $(X,\mu,T)$.
\end{theorem}

In~\cite{Z15b}, an explicit construction of an automorphism with a given scaling sequence is provided via \emph{the adic transformation} on \emph{the graph of ordered pairs} and special \emph{central measures} on this graph (see also papers~\cite{VZ17, VZ18a, VZ18b} and Section~\ref{section_adic} of this survey). Theorem~\ref{theorem_example_entropy_sequence} gives a rich family of automorphisms with given entropy properties and is one of the key results in the theory of scaling entropy as well as a useful tool in applications.  

Note that Theorem~\ref{theorem_example_entropy_sequence} shows a significant difference between the scaling entropy and the complexity of a topological dynamical system: complexity function~$p(n)$ is always either bounded or $p(n) \ge n$ while the scaling entropy can have an arbitrary given asymptotic behavior. However, the inequality {$\scfunclass \preceq [\log p(n)]$}  holds for any invariant measure.

In~\cite{FP,KT} possible values of similar invariants of slow entropy type were studied. We discuss relations between these invariants and the scaling entropy in Section~\ref{section_slow_entropy}.

We denote by~$\Subadd$ the set of all equivalence classes of subadditive non-decreasing sequences with respect to equivalence relation~$\asymp$. Relation~$\preceq$ is naturally defined on such classes and forms a partial order on~$\Subadd$ making it an upper semilattice.

The least element of~$\Subadd$ is the equivalence class of a constant sequence~$h_n=1$, and the greatest element is the equivalence class of a linear function~$h_n=n$. The following two theorems proved in~\cite{V11,VPZ13a,Z15a} give the description of dynamical systems whose scaling sequence attends these extreme values. 

\begin{theorem}\label{th_diskr_spectr}
Let $T$ be an automorphism of a measure space $(X,\mu)$. The following are equivalent: 
\begin{enumerate}
 \item $T$ has pure point spectrum;
 \item The sequence $h_n=1$ is a scaling sequence for $(X,\mu,T)$;
 \item There exists a $T$-invariant admissible metric $\rho \in \Adm(X,\mu)$.
\end{enumerate}

\end{theorem}

\begin{theorem}
The sequence $h_n=n$  is a scaling sequence for $(X,\mu,T)$ if and only if Kolmogorov entropy is positive: $h(T)>0$.
\end{theorem}

In particular, the systems with the maximal and minimal growth of scaling entropy are stable.

\section{Properties of scaling entropy}

\subsection{Example of a non-stable system}\label{section_nonstable_example}
For some time, the question of the existence of a scaling sequence for any {ergodic automorphism~$T$ of a standard probability space~$(X, \mu)$ remained open.} An example of such an ergodic automorphism~$T$ and an admissible metric~$\rho$ on~$(X, \mu)$, for which the scaling sequence in the sense of Definition~\ref{def_sc_ent} does not exist, {was constructed in~\cite{Vep20a}}. The reason for this phenomenon is that for different $\eps>0$, the growth rate of $\heps_\eps(X, \mu, T_{av}^n\rho)$ with respect to $n$ can significantly differ: for each $\eps>0$, there exists $\delta>0$ such that
$$
\lim_{n \to +\infty} \frac{\heps_\eps(X,\mu,T_{av}^n\rho)}{\heps_\delta(X,\mu,T_{av}^n\rho)} = 0.
$$

However, the construction of such an automorphism requires the existence of stable systems that satisfy Theorem~\ref{theorem_example_entropy_sequence}. Let us select a family $h^{(k)} = {h_n^{(k)}}_{n \in \mathbb{N}}$, $k \in \mathbb{N}$ of increasing subadditive sequences in such a way that for each~$k$ the following holds:
$$
h_n^{(k)} = o(h_n^{(k+1)}), \qquad n \to \infty.
$$
For each $k\in \mathbb{N}$, find an ergodic automorphism $T_k$ of a probability space $(X_k, \mu_k)$ such that $h^{(k)} \in \scclass(X_k,\mu_k,T_k)$. Such an automorphism exists due to Theorem~\ref{theorem_example_entropy_sequence}. The following theorem proved in~\cite{Vep20a} guarantees the existence of unstable automorphisms.
\begin{theorem}
Let $(X,\mu,T)$ be an ergodic joining of a family of systems $(X_k,\mu_k,T_k)$, $k \in \mathbb{N}$. Then, the class $\scclass(X,\mu,T)$ of scaling entropy sequences for the system $(X,\mu,T)$ is empty.
\end{theorem}

The main idea of the proof is that if the system $(X,\mu,T)$ has a scaling sequence $h$, then this sequence would be the least upper bound for the family of sequences $h^{(k)}$, $k \in \mathbb{N}$. However, the set $\Subadd$ does not contain the least upper bound for a strictly increasing sequence of elements, therefore, the class $\scclass(X,\mu,T)$ is empty.
{This example highlights the meaning of Definition~\ref{definition_asymp_class}: the asymptotic class of a function $\Phi$ can be identified with the least upper bound of the classes of sequences $\Phi(\eps, \,\cdot\,)$.}

\subsection{Possible values of scaling entropy. The semilattice of functions}
\label{section_scentropy_values}
{In this section, we study the asymptotic classes that could be the scaling entropy of dynamical systems.} The following two theorems, proved in paper~\cite{Vep20a}, provide a complete description of the possible values of scaling entropy.
\begin{theorem}\label{th_mon_subadd}
 For any system $(X,\mu,T)$, in the scaling entropy class $\scfunclass(X,\mu,T)$ one can always find a function $\Phi \colon \mathbb{R}_+\times \mathbb{N} \to \mathbb{R}_+$ with the following properties:
 \begin{enumerate}
     \item $\Phi(\,\cdot\,, n)$ is non-increasing for any $n \in \mathbb{N}$;
     \item $\Phi(\eps, \,\cdot\,)$ is non-decreasing and subadditive for any $\eps>0$.
 \end{enumerate}
\end{theorem}

\begin{theorem}\label{theorem_complete_description}
For any function $\Phi \colon \mathbb{R}_+\times \mathbb{N} \to \mathbb{R}_+$ satisfying properties 1) and 2) from the previous theorem, there exists an ergodic system $(X,\mu,T)$ such that $\Phi \in \scfunclass(X,\mu,T)$.
\end{theorem}

The proofs of Theorems~\ref{th_mon_subadd} and~\ref{theorem_complete_description} crucially rely on the corresponding results for the stable case (Theorems~\ref{theorem_subadd_entropy_sequence} and~\ref{theorem_example_entropy_sequence}). It is worth noting that the analog of Theorem~\ref{theorem_complete_description} for actions of amenable groups (see Section~\ref{section_groups}) is unknown to the authors, the question of describing the set of possible values of scaling entropy for group actions remains open.

One can consider the partially ordered set of equivalence classes of all two-variable functions satisfying the monotonicity conditions 1) and 2) from Theorem~\ref{th_mon_subadd}. This set forms an upper semilattice. The semilattice $\Subadd$ of subadditive non-decreasing sequences is naturally embedded into this lattice. Unlike the semilattice $\Subadd$, the considered semilattice of functions possesses the following property: any countable subset has the least upper bound. Moreover, this semilattice of functions is the minimal semilattice {containing $\Subadd$} with this property, as each scaling entropy $[\Phi(\,\cdot\,,\,\cdot\,)]$ is the least upper bound for the countable set of sequences $h^m = {\Phi(\frac{1}{m}, n)}_n$, $m \in \mathbb{N}$.

\begin{proposition}
Let $(X_k,\mu_k,T_k)$ be a (finite or countable) sequence of systems, and $(X,\mu,T)$ be their joining. Then $\scfunclass(X,\mu,T)$ is the least upper bound of the sequence $\scfunclass(X_k,\mu_k,T_k)$.
\end{proposition}

\subsection{Generic scaling entropy}
The group $\aut(X,\mu)$ of all automorphisms of the space $(X,\mu)$ equipped with the weak topology is a Polish topological space. This allows us to investigate the genericity of automorphisms satisfying given properties. It turns out that the scaling entropy of a generic automorphism $T\in \aut(X,\mu)$ is not comparable to an arbitrary given function (except for trivial extreme cases of bounded and linear growth). The following theorem was proved in~\cite{Vep21}.

\begin{theorem}\label{theorem_generic_transformation}
Let $\Phi(\eps,n)$ be a function that {decreases in $\eps$ for each $n$, and for each $\eps>0$ is increasing, unbounded, and sublinear in $n$}. Then the set of automorphisms whose scaling entropy is not comparable to $\Phi$ is comeager in~$\aut(X,\mu)$.
\end{theorem}

{
\begin{remark}
     Theorem~\ref{theorem_generic_transformation} can be reformulated as two statements: the set of automorphisms $T$ satisfying $\scfunclass(T) \prec \Phi$ is meager, and the set of automorphisms~$T$ satisfying $\scfunclass(T) \succ \Phi$ is also meager.
\end{remark}
The proof of Theorem~\ref{theorem_generic_transformation} uses the connection between scaling entropy and Kirillov--Kushnirenko sequential entropy (see~\cite{Kush}) and the results from paper~\cite{R} on the genericity of infinite Kirillov--Kushnirenko entropy. We discuss the connection between scaling entropy and sequential entropy in Section~\ref{section_slow_entropy}. Let us mention that related results on generic values of related invariants were obtained in~\cite{A,AGTW}.
}

\subsection{Scaling entropy and ergodic decomposition}
The scaling entropy of a non-ergodic system can grow faster than the scaling entropy of all ergodic components. {Indeed, consider a classical example of a non-ergodic transformation of the torus (see, for instance, \cite{Kush}): $(x,y)\mapsto (x,y+x)$, where $x, y \in \mathbb{T}^2$}. The ergodic components of this transformation are rotations of the circle, which have bounded scaling entropy. However, the system itself has unbounded scaling entropy $\scfunclass = [\log n]$. Nevertheless, there exists an estimate in the opposite direction.

\begin{proposition}\label{ergodic_decomosition}
 Let $(X,\mu,T)$ be a dynamical system, and $\mu = \int \mu_\alpha d \nu(\alpha)$ its decomposition into ergodic components. Let the sequence $h = (h_n)$ be such that for any $\alpha$ from some set of positive measure, the following holds:
 $$\mathcal{H}(X,\mu_\alpha,T) \succeq h_n.$$
  Then {$\mathcal{H}(X,\mu,T) \succeq h_n$}.
\end{proposition}

The proof of Proposition~\ref{ergodic_decomosition} is given in Appendix~\ref{section_proof_erg_dec}.

\section{Examples: computing the scaling entropy}

In this section, we present several results on the explicit computation of the scaling entropy for several automorphisms. It's worth mentioning that finding the scaling entropy in general can be a challenging computation. We state several open questions on computing the scaling entropy for some well-known transformations, such as the Pascal automorphism (see~\cite{VPasc}), in Appendix~\ref{section_problems}.

We mentioned in Section~\ref{section_scaling_entropy_sequence_values} that transformations with positive Kolmogorov entropy and only them have scaling entropy $\scfunclass = [n]$, while transformations with pure point spectrum and only them have bounded scaling entropy, i.\,e., $\scfunclass = [1]$.

\subsection{Substitution dynamical systems}\label{Sec_podst}

Let $A$ be an alphabet of finite size, $|A| > 1$. We denote by $A^*$ the set of all words of finite length over the alphabet $A$. \emph{A substitution} is an arbitrary mapping $\xi\colon A \to A^*$. The mapping $\xi$ naturally extends to a mapping $\xi\colon A^* \to A^*$ and, moreover, to a mapping $\xi\colon A^{\mathbb{N}} \to A^{\mathbb{N}}$, where $A^{\mathbb{N}}$ is the space of one-sided sequences of elements from the set $A$, equipped with the standard product topology. We will assume that the substitution $\xi$ is such that there exists an infinite word $u\in A^{\mathbb{N}}$ invariant under $\xi$: $\xi(u) = u$. Let $T\colon A^{\mathbb{N}} \to A^{\mathbb{N}}$ be the left shift. A substitution dynamical system is defined as the pair $(X_\xi, T)$, where $X_\xi$ is the closure of the orbit of the point $u$ under the action of the transformation $T$. A substitution is called \emph{primitive} if for some $n \in \mathbb{N}$ and any $\alpha, \beta \in A$, the letter $\beta$ appears in the word $\xi^n(\alpha)$. If the substitution $\xi$ is primitive then there exists a unique $T$-invariant Borel probability measure $\mu^\xi$ on the compact topological space $X_\xi$.

 We say that $\xi$ is a \emph{substitution of constant length} if there exists a natural number $q \in \mathbb{N}$ such that $|\xi(\alpha)| = q$ for any $\alpha \in A$. The \emph{height $h(\xi)$ of a substitution} is defined as the largest natural number $k$ coprime with $q$ such that if $u_n = u_0$, then $k \,|\, n$. \emph{The column number} $c(\xi)$ is defined as follows:
$$
c(\xi) = \min \Big\{\big|\{\ \xi^k(\alpha)_i \colon \alpha \in A  \}\big| \colon k \in \mathbb{N}, i < q^k \Big\}.
$$
For more details about substitution dynamical systems see, e.\,g.,~\cite{Que}. The scaling entropy of a substitution dynamical system corresponding to a substitution of constant length was computed in~\cite{Z15a}.

\begin{theorem}\label{theorem_example_substitution}
   Let $\xi$ be an injective primitive substitution of constant length. Then $\scfunclass(X_\xi, \mu^\xi, T) = [\log n]$ if $c(\xi) \not = h(\xi)$, and $\scfunclass(X_\xi, \mu^\xi, T) = [1]$, if $c(\xi) = h(\xi)$.

\end{theorem}

One of the special cases of substitution systems described by Theorem~\ref{theorem_example_substitution} is the Morse automorphism.
    
\begin{corollary}
    The Morse automorphism has scaling entropy $\scfunclass(T) = [\log n]$.
\end{corollary}

The Chacon automorphism is not a substitution of constant length, however, it also exhibits logarithmic scaling entropy. The following theorem follows from the results from~\cite{F}.

\begin{theorem}
    The Chacon automorphism has scaling entropy $\scfunclass(T) = [\log n]$.
\end{theorem}

Connections between substitutions and stationary adic transformations were studied in~\cite{VL92}. There one can also find examples of adic realizations of substitution dynamical systems, including the Chacon automorphism.

\subsection{Horocycle flows}

Another example of a classical automorphism with logarithmic scaling entropy is a horocycle flow.
Related entropy-like invariants for classical flows were studied in~\cite{Kani, KaniVinWei, Kush}. The following theorem follows from the results from~\cite{KaniVinWei}.

\begin{theorem}
The horocycle flow on a compact surface with constant negative curvature has scaling entropy $\scfunclass(T) = [\log n]$.
\end{theorem}

It is noteworthy that despite the significant difference in scaling entropy, the horocycle flow and the Bernoulli automorphism share the same multiple Lebesgue spectrum.

Also, note that for transformations with logarithmic scaling entropy, it makes sense to compute a refinement of our invariant - the exponential scaling entropy (see Section~\ref{section_exp_scaling_entropy}).

\subsection{Adic transformation on the graph of ordered pairs}\label{section_adic}
	\newcommand{\edgemark}{\mathfrak c}

       Adic transformations (Vershik automorphisms), including the adic transformation on the graph of ordered pairs, were studied~\cite{V81, V82, VZ17, VZ18a, VZ18b, Z15b, Vep20b}. The construction presented below is key to proving Theorems~\ref{theorem_example_entropy_sequence} and~\ref{theorem_complete_description}. It provides an explicit realization of any subadditive and increasing scaling entropy sequence.

        Consider an infinite graded graph $\Gamma = (V, E)$. The set of vertices $V$ of the graph $\Gamma$ is a disjoint union of sets $V_n = \{0,1\}^{2^n}$, where $n \geq 0$. The set of edges $E$ is defined together with the coloring $\edgemark \colon E \to \{0,1\}$ as follows. Let $v_n \in V_n$ and $v_{n+1} \in V_{n+1}$. An edge $e = (v_n, v_{n+1})$ belongs to $E$ if the word $v_n$ is a prefix or suffix of the word $v_{n+1}$ and is labeled with the symbol $0$ or $1$ respectively.

        A Borel measure on the space $X$ of all infinite paths in the graph $\Gamma$ is called \emph{central} if, given a fixed tail of a path, all its starting points are equiprobable.

        Let us define \emph{the adic transformation} $T$ on the space of paths $X$. Let $x = \{e_i\}_{i=1}^\infty$ be an infinite path. Find the smallest $n$ such that $\edgemark(e_n) = 0$. Define the path $T(x) = \{u_i\}$ as follows. For $i \ge n+1$, we have $u_i = e_i$; $\edgemark(u_n) = 1$, and $\edgemark(u_i) = 0$ for all $i < n$. For any central measure $\mu$, the transformation $T$ is an automorphism of the measure space $(X, \mu)$.

        Let us fix a certain sequence $\sigma = \{\sigma_n\}$ consisting of zeros and ones. We will construct the corresponding central measure $\mu^\sigma$ on the space $X$. A Borel measure $\mu$ on the space $X$ is uniquely determined by an agreeing system of measures $\mu_n$ on cylindrical sets corresponding to finite paths of length~$n$. In terms of $\mu_n$, the centrality of the measure $\mu$ means that for any $n$, the measure $\mu_n$ depends only on the endpoint of the path. Let $\nu_n$ be the projection of $\mu_n$ onto the vertex set $V_n$, corresponding to the endpoint of the path. The system of measures $\nu_n$ uniquely determines the central measure~$\mu$.

        Let us construct a sequence of sets $V_n^\sigma$, where $V_n^\sigma \subset V_n$. We set $V_0^\sigma = V_0$. For $n \ge 1$, we define

        $$
            V_n^\sigma = 
            \begin{cases}
             \{ab \colon a,b \in V_{n-1}^\sigma\}, \quad & \text{if } \sigma_n = 1;\\
             \{aa \colon a \in V_{n-1}^\sigma\}, \quad & \text{if } \sigma_n = 0. 
            \end{cases}
        $$
        Let $\nu_n^\sigma$ denote the uniform measure on the set $V_n^\sigma \subset V_n$. The measure $\mu^\sigma$ constructed based on this system is defined correctly and is central.

\begin{theorem}
    The adic transformation on the paths of the graph of ordered pairs with measure $\mu^\sigma$ has scaling entropy $\scfunclass(T) = [2 ^{\sum_{i=0}^{\log n} \sigma_i}]$.
\end{theorem}

\section{Scaling entropy of a group action}\label{section_groups}
In this section, we present some introductory facts and examples related to the generalization of the notion of scaling entropy for actions of general discrete groups. The classical entropy theory can be to a large extent applied to actions of amenable groups (see~\cite{OW87}).
The aforementioned theory of scaling entropy can be only partially extended from the case of a single automorphism to group actions. Most of the results we will discuss in this section deal with amenable groups. However, we provide a definition of scaling entropy for actions of arbitrary countable groups which is invariant under the change of admissible generating semimetric.
The scaling entropy for group actions was studied in~\cite{Z15b, PZ15, Vep20b, Vep22}. Related invariants of p.m.p. actions of amenable groups were also studied~\cite{KT, Lo}.

\subsection{Definition of scaling entropy of a group action}

Consider a countable group $G$ acting by automorphisms on the standard probability space $(X,\mu)$.

\begin{definition}
\emph{Equipment} of a countable group $G$ is a sequence $\sigma = \{G_n\}_{n \in \mathbb{N}}$ of finite subsets of the group, for which $|G_n| \to +\infty$. We denote by $(G,\sigma)$ a group with chosen equipment.
\end{definition}

For a semimetric $\rho$ on $(X,\mu)$ and a subset $H \subset G$, we denote by $H^{av} \rho$ the averaging of the semimetric $\rho$ over the shifts by elements $g \in H$:

$$
H^{av} \rho (x,y) = \frac{1}{|H|}\sum_{g \in H} \rho(gx,gy).
$$

For an action of an equipped group $G$ with equipment $\sigma=\{G_n\}$ on the standard probability space $(X,\mu)$ and a semimetric $\rho \in \Adm(X,\mu)$, we define the function $\Phi_\rho$ on $\mathbb{R}_+\times \mathbb{N}$ similarly to formula~\eqref{Phi_metr}:
$$
\Phi_\rho(\eps, n) = \heps_\eps(X,\mu,G_n^{av}\rho).
$$

For group actions, the cases of admissible metrics and semimetrics differ. For averages of admissible \emph{metrics}, an analog of Lemma~\ref{lem_asymp_ineq}, which was proved in~\cite{Z15b}, holds. The corresponding statement for semimetrics (Theorem~\ref{theorem_ivariance_groups}) will appear later.
 
\begin{lemma}
Let $\rho_1,\rho_2 \in \Adm(X,\mu)$. If $\rho_1$ is a metric, then for any $\epsilon>0$, there exists a $\delta>0$ such that
$$
\Phi_{\rho_2}(\eps, n) \preceq \Phi_{\rho_1}(\delta, n), \qquad n \to +\infty. 
$$
In other words,
$$
\Phi_{\rho_2} \preceq \Phi_{\rho_1}.
$$
\end{lemma}
\begin{corollary}\label{cor2}
If $\rho_1,\rho_2 \in \Adm(X,\mu)$ are metrics, then $\Phi_{\rho_1} \asymp \Phi_{\rho_2}$.
\end{corollary}

Thus, as before, the scaling entropy of an action of an equipped group $(G,\sigma)$ on $(X,\mu)$ can be defined:
\begin{definition}
\emph{The scaling entropy} of an action of an equipped group $(G,\sigma)$ on $(X,\mu)$ is the equivalence class $[\Phi_{\rho}]$ for some (hence for any) metric $\rho \in \Adm(X,\mu)$. We will denote this class by $\scfunclass(X,\mu,G, \sigma)$.
\end{definition}
Let us emphasize that the asymptotic class $\scfunclass(X,\mu, G, \sigma)$ is a measure-theoretic invariant of an action.

\begin{definition}
We say that equipment $\sigma = \{G_n\}_{n \in \mathbb{N}}$ is \emph{suitable} if for any $g \in \cup G_n$ and any $\delta>0$, there exists a number $k \in \mathbb{N}$ such that for each $n$, there exist elements $g_1, \dots, g_k \in G$, such that
$$
\Big| gG_n \setminus \bigcup_{j=1}^{k}G_ng_j \Big| \leq \delta |G_n|.
$$
\end{definition}

\begin{remark}
Equipment is always suitable if it is
\begin{itemize}
\item equipment of a commutative group;
\item a F\o lner sequence of an amenable group;
\item a sequence of balls in a finitely generated group;
\item a sequence of finite expanding subgroups of any group.
\end{itemize}
Note that not every group has suitable equipment by subsets that together generate the whole group. For instance, the free group $F_A$ over an infinite alphabet $A$ does not have such equipment.
\end{remark}

We say that a semimetric $\rho \in \Adm(X,\mu)$ is \emph{generating for the action of an equipped group} $(G,\sigma)$ (or $G$-\emph{generating}) if its translations $g^{-1}\rho, \ g \in \cup G_n$ by the elements from the union of equipment $\sigma$ separate points of some subset of full measure. In the case of an amenable group $G$, equipped with a F\o lner sequence, we call a semimetric $\rho$ \emph{generating}, if all of its translations together separate points of some subset of full measure. The following theorem was proved in~\cite{Z15b}.

\begin{theorem}\label{theorem_ivariance_groups}
Let $\sigma = \{G_n\}$ be suitable equipment of a group~$G$. Let $\rho_1,\rho_2 \in \Adm(X,\mu)$. If $\rho_1$ is a $G$-generating semimetric, then
$$
\Phi_{\rho_2} \preceq \Phi_{\rho_1}.
$$
In particular,
$$[\Phi_{\rho_1}] = \scfunclass(X,\mu,G,\sigma).$$
\end{theorem}

\subsection{Properties of scaling entropy of a group action}

The natural question is whether the scaling entropy of a group action depends on the choice of equipment. A simple observation shows that a small change of equipment does not change the scaling entropy.
\begin{remark}
If equipment $\sigma_1 = {G_n^{(1)}}$ and equipment $\sigma_2 = {G_n^{(2)}}$ of a group $G$ are such that
$$|G_n^{(1)}\Delta G_n^{(2)}| = o(|G_n^{(1)}|), \qquad n \to +\infty,$$
then the scaling entropies of actions of $G$ equipped with $\sigma_1$ and $\sigma_2$ coincide:
$$
\scfunclass(X,\mu,G,\sigma_1) = \scfunclass(X,\mu,G,\sigma_2).
$$
\end{remark}

There is a simple upper bound for the scaling entropy.
\begin{theorem}
Let a group $G$ with equipment $\sigma = \{G_n\}$ act by automorphisms on a measure space $(X,\mu)$. Then for $\Phi \in \scfunclass(X,\mu,G,\sigma)$, for every $\eps>0$, the following inequality holds:
$$
\Phi(\eps, n) \preceq |G_n|, \qquad n \to +\infty.
$$
\end{theorem}

For the case of an amenable group $G$ with F\o lner equipment $\sigma$, the previous theorem admits a refinement.
\begin{theorem}
Let an amenable group $G$ equipped by a F\o lner sequence $\sigma = \{G_n\}$ act by automorphisms on a measure space $(X,\mu)$. For $\Phi \in \scfunclass(X,\mu,G,\sigma)$, for any $\eps>0$, the asymptotic relation 
$$
\Phi(\eps, n)  = o(|G_n|), \qquad n \to +\infty,
$$
holds if and only if the Kolmogorov entropy of the action is zero.
\end{theorem}

For a finitely generated group, the existence of a compact free action is equivalent to the group being residually finite. If the group is also amenable, then the compactness of the action is equivalent to the boundedness of the scaling entropy. The following generalization of Theorem~\ref{th_diskr_spectr} was proved in~\cite{YZZ}.
\begin{theorem}
    Let an amenable group $G$ equipped by a F\o lner sequence $\sigma = \{G_n\}$ act by automorphisms on a measure space $(X,\mu)$. Then $\scfunclass(X,\mu,G,\sigma) = [1]$ if and only if the action of $G$ on $(X,\mu)$ is compact.
\end{theorem}

{
\subsection{Scaling entropy of a generic action}
For a given countable group $G$, the set of all its p.m.p. actions $A(X, \mu, G)$ on a Lebesgue space $(X,\mu)$ forms a Polish topological space, allowing us to discuss generic properties of actions of $G$. For more details on the theory of generic group actions see a survey~\cite{Ke}. The following theorem, proved in~\cite{Vep22}, generalizes a similar result concerning the absence of nontrivial upper bounds for the scaling entropy of a generic automorphism to the case of an arbitrary amenable group.
\begin{theorem}\label{theorem_genegic_upper_bound}
Let $G$ be an amenable group and $\sigma = {F_n}$ be its F\o lner sequence. Let $\phi(n) = o(|F_n|)$ be a sequence of positive numbers. Then the set of actions $\alpha \in A(X, \mu, G)$ for which $\mathcal{H}(\alpha, \sigma) \not\prec \phi$ contains a dense $G_\delta$ subset.
\end{theorem}
The proof of Theorem~\ref{theorem_genegic_upper_bound} follows with additional adjustments the proof of Theorem~\ref{theorem_generic_transformation} and uses the results from~\cite{R}. Related results for similar invariants in the context of generic extensions were obtained in~\cite{Lo}.

Note that a direct analog of Theorem~\ref{theorem_complete_description} that is a complete description of possible values of scaling entropy for actions of a group is not known to the authors. However, a weak version of this theorem follows from Theorem~\ref{theorem_genegic_upper_bound}: for any sequence $\phi(n) = o(|F_n|)$, there exists an ergodic action of the group $G$ with scaling entropy $\scfunclass$ that grows faster than $\phi$ along some subsequence. For non-periodic amenable groups, explicit constructions of such p.m.p. actions can be obtained using coinduction from an action of a subgroup to the action of the ambient group~(see~\cite{Vep20b}). Also, note that explicit constructions of such actions for arbitrary amenable groups are not known to the authors.

Unlike nontrivial upper bounds for the scaling entropy of a generic system, the absence of lower bounds requires certain conditions on the group. In particular, a sufficient condition is the existence of a compact free action of the group. Thus, the following theorem which proved in~\cite{Vep22} holds.
\begin{theorem}\label{theorem_genegic_lower_bound}
Let $G$ be a residually finite amenable group and $\sigma = {F_n}$ be its F\o lner sequence. Let $\phi(n)$ be a sequence of positive numbers increasing to infinity. Then the set of actions $\alpha \in A(X, \mu, G)$ for which $\mathcal{H}(\alpha, \sigma) \not\succ \phi$ contains a dense $G_\delta$-subset.
\end{theorem}

We will show in Section~\ref{section_entropy_gap} that in order for a group~$G$ to have no nontrivial lower bounds on the scaling entropy of a generic action, it is \emph{necessary} to impose certain conditions on the group $G$.

\subsection{Scaling entropy growth gap}\label{section_entropy_gap}

The following theorem proved in~\cite{Vep22} shows that there exist non-residually finite amenable groups for which the conclusion of Theorem~\ref{theorem_genegic_lower_bound} is not true.
\begin{theorem}\label{theorem_gap}
Let $G = SL(2, \overline{\mathbb{F}}_p)$ be the group of all $2\times2$ matrices over the algebraic closure of the finite field $\mathbb{F}_p$, where $p > 2$. Let $\mathbb{F}_{p} = \mathbb{F}_{q_0} \subset \mathbb{F}_{q_1} \subset \ldots$ be a sequence of finite extensions that together cover the whole  $\overline{\mathbb{F}}_p$, and let $\sigma = \{SL(2, \mathbb{F}_{q_n})\}$ be equipment of the group~$G$ by a sequence of increasing finite subgroups. Then, any free action $\alpha \in A(X, \mu, G)$ satisfies  $\mathcal{H}(\alpha, \sigma) \succsim \log q_n$.
\end{theorem}

The proof of Theorem~\ref{theorem_gap} is based on the theory of growth in finite groups $SL(2, \mathbb{F}_{q_n})$, in particular, on  Helfgott's theorem and its generalizations (see \cite{Hel, PSz}), as well as the representation theory of these groups (see \cite{Sch, Jor}).

\begin{definition}
    We say that an amenable group $G$ equipped by a F\o lner sequence $\sigma$ \emph{has a scaling entropy growth gap} if there exists an increasing sequence~$\phi(n)$ tending to infinity such that for any free action~$\alpha$ of~$G$, we have $\mathcal{H}(\alpha, \sigma) \succsim \phi(n)$.
\end{definition}

\begin{proposition}
 For an amenable group, the property of having a scaling entropy growth gap does not depend on the choice of a F\o lner sequence.
\end{proposition}
}
Thus, the property of having a scaling entropy growth gap is a group property of an amenable group. It is a natural question to verify this property for particular amenable groups. In a paper in progress, the second author proves that the infinite symmetric group has a scaling entropy growth gap. He also proves that this property is preserved when passing from a subgroup to the ambient group. In particular, it is shown that there are finitely generating amenable groups that have a scaling entropy growth gap. In general, the question of classifying groups with a scaling entropy growth gap is widely open. For example, it is unknown to the authors if finitely generated simple amenable groups (see, for example,~\cite{JM}) satisfy this property.

\section{Universal zero entropy system problem}
{Universal systems in various contexts have been studied by many authors in lots of papers, see, for example,~\cite{DS, S, VZ17, VZ18a, VZ18b, Vep20a, Vep22}. We will follow the definition proposed in~\cite{DS, S}. Let $\mathcal{S}$ be a certain class of p.m.p. actions of an amenable group $G$. A topological system $(X, G)$ is called \emph{universal} for the class~$\mathcal{S}$ if, for any invariant measure $\mu$ on $X$, the system $(X, \mu, G)$ belongs to $\mathcal{S}$ and conversely, any system from the class $\mathcal{S}$ can be realized using some invariant measure $\mu$ on $X$. In~\cite{S}, Serafin addresses the question, going back to Weiss, about the existence of a universal dynamical system for the class $\mathcal{S}$ consisting of all actions with zero measure-theoretic entropy. In~\cite{S}, the negative answer is given for the case $G = \mathbb{Z}$. In his work, the author points out that his approach, based on the theory of symbolic coding and the theory of algorithmic complexity, did not yield the desired result for arbitrary amenable groups.

The theory of scaling entropy allows us to give a negative answer to Weiss's question for all amenable groups. The following result was obtained in papers~\cite{Vep20b, Vep22}.
\begin{theorem}\label{theorem_universal_system}
Any infinite amenable group~$G$ does not admit a universal zero entropy system.
\end{theorem}
The main role in proving this result is played by a special series of group actions that satisfy certain conditions on the growth of scaling entropy.

\begin{definition}
We say that a group $G$ with equipment $\sigma = {G_n}$ \emph{admits actions of almost complete growth} if for any non-negative function $\phi(n) = o (|G_n|)$, there exists an ergodic system $(X, \mu, G)$ such that for any $\Phi \in \mathcal{H} (X, \mu, G, \sigma)$ and any sufficiently small $\eps > 0$, the following relations hold:
$$
\Phi(\eps, n) \not \preceq \phi(n) \text{ and } \Phi(\eps, n) = o (|G_n|).
$$
\end{definition}
The existence of such actions is a sufficient condition for the absence of a universal system of zero entropy.
Actions of almost complete growth for a non-periodic amenable group and any F\o lner sequence can be constructed explicitly {(see \cite{Vep20b})}, using Vershik's automorphisms on the graph of ordered pairs {(see \cite{VZ17, VZ18a, VZ18b})} and the coinduction operation from a subgroup $\mathbb{Z}$ to the entire group $G$. For arbitrary amenable groups, explicit constructions of such actions are not known to the authors. However, Theorem~\ref{theorem_genegic_upper_bound} guarantees the genericity of actions of almost full growth and, hence, their existence in the general case.

\begin{theorem}
Any infinite amenable group~$G$ equipped with a F\o lner sequence~$\sigma = \{G_n\}$ admits actions of almost complete growth. 
\end{theorem}

\section{Exponential scaling entropy and other related invariants}

In this chapter, we will discuss several invariants related to scaling entropy.

\subsection{Exponential scaling entropy}\label{section_exp_scaling_entropy}

It turns out that Lemma~\ref{lem_asymp_ineq}, which plays a crucial role in the theory of scaling entropy, can be refined in the following way. A similar asymptotic relation holds not only for the function $\Phi_\rho(\eps,n) = \heps_\eps(x,\mu, T_{av}^n\rho)$, but also for $\exp(\Phi_\rho(\eps,n))$, that is, for the size of the minimal $\eps$-net of the semimetric $T_{av}^n\rho$ on the set of measure $1-\eps$, rather than for its logarithm (see Definition~\ref{definition_epsilon_entropy}).
\begin{lemma}\label{lem_asymp_ineq_exp}
Let $\rho_1,\rho_2 \in \Adm(X,\mu)$. If $\rho_1$ is $T$-generating, then for any $\eps>0$, there exists $\delta>0$ such that:
\begin{equation}\label{equation_asymp_ineq_exp}
\exp(\Phi_{\rho_2}(\eps,n)) \preceq \exp(\Phi_{\rho_1}(\delta,n)), \qquad n \to +\infty. 
\end{equation}
\end{lemma}

\begin{proof}
First, we show that relation~\eqref{equation_asymp_ineq_exp} holds for the semimetric $\rho_2 = T_{av}^k\rho_1$.
\begin{lemma}\label{lemma_asymp_average}
    For any positive integer number $k$ and positive $\eps$, there exists a positive integer $N$ such that:
    \begin{equation}
        \heps_\eps(X,\mu, T_{av}^n(T_{av}^k\rho_1)) \le \heps_{\eps/4}(X,\mu, T_{av}^n\rho_1), \quad n > N.
    \end{equation}
\end{lemma}
\begin{proof}
    Indeed
    \begin{equation}\label{eq270201}
      T_{av}^n(T_{av}^k\rho_1) (x,y) \le T_{av}^n\rho_1(x,y) + \frac{1}{n} \sum\limits_{i=n}^{n+k-1} T^{-i}\rho_1(x,y).  
    \end{equation}
    
    The last term in the right-hand side of equation~\eqref{eq270201} is bounded in $m$-norm by $\frac{k}{n}||\rho_1||_m$. Therefore, by Lemma~\ref{lemma_semicontinuity}, for sufficiently large $n$,
    $$
    \heps_\eps\big(X,\mu, T_{av}^n(T_{av}^k\rho_1)\big) \le \heps_\eps\big(X,\mu, T_{av}^n\rho_1 + \frac{1}{n} \sum_{i=n}^{n+k-1} T^{-i}\rho_1\big) \le \heps_{\eps/4}\big(X,\mu, T_{av}^n\rho_1\big).
    $$
\end{proof}
Note that the proof of Lemma~\ref{lemma_asymp_average} works without any changes for the case of an amenable group equipped with a F\o lner sequence.

Lemma~\ref{lemma_asymp_average} ensures that the relation~\eqref{equation_asymp_ineq_exp} holds for admissible metric
$$
\rho_2 = \rho = \sum_{i=0}^\infty \frac{1}{2^i} T^{-i} \rho_1.
$$
Next, we follow with certain refinements the arguments from~\cite{Z15a}. The set $\mathcal{M}$ of all semimetrics $\rho_2 \in \Adm(X,\mu)$ that satisfy the relation~\eqref{equation_asymp_ineq_exp} is closed in the $m$-norm due to Lemma~\ref{lemma_semicontinuity}. Let us show that the set $\mathcal{\tilde M}$ of all semimetrics $\omega$ that satisfy the following inequality for all $x,y\in X$
$$
\omega(x,y) \le C(\omega) \rho(x,y),
$$
is dense in $(\Adm(X,\mu), \|\,\cdot\,\|_m)$. Indeed, as shown in \cite{Z15a}, any admissible integrable semimetric can be approximated in $m$-norm by semimetrics which can be dominated by finite sums of cut semimetrics. Any cut semimetric can be approximated by a semimetric of the form $df = |f(x) - f(y)|$ for some function $f \in L^1(X,\mu)$ which is Lipschitz with respect to the metric $\rho$. Clearly,  $d[f] \in \mathcal{\tilde M}$, therefore, $\mathcal{\tilde M}$ is dense in $(\Adm(X,\mu), \|\,\cdot\,\|_m)$. However, $\mathcal{\tilde M} \subset \mathcal{M}$, implying that $\mathcal{M} = \Adm(X,\mu)$.
\end{proof}

Lemma \ref{lem_asymp_ineq_exp}, as before, ensures the independence of the class $[\exp(\Phi_\rho)]$ from a $T$-generating semimetric $\rho$ and allows us to give the following definition.

\begin{definition}
\emph{The exponential scaling entropy} of the system $(X,\mu,T)$ is the equivalence class $[\exp(\Phi_\rho)]$ for some (hence for any) $T$-generating semimetric $\rho \in \Adm(X,\mu)$. We will denote this equivalence class by $\scfunclass_{\exp}(X,\mu,T)$.
\end{definition}

Note that Lemma~\ref{lem_asymp_ineq_exp} holds for p.m.p. actions of amenable groups equipped with a F\o lner sequence. For an action $(X,\mu, G)$ of an amenable group $G$ with a F\o lner sequence $\lambda = {F_n}$, we define its \emph{exponential scaling entropy} $\scfunclass_{\exp}(X,\mu,G,\lambda)$ as the equivalence class of the function $\exp(\Phi_\rho(\eps,n))= \exp(\heps_\eps(x,\mu, G_{av}^n\rho))$ for some (hence for any) generating semimetric $\rho$.

\emph{Exponential scaling entropy $\scfunclass_{\exp}(X,\mu, T)$ is a finer invariant of a dynamical system than the usual scaling entropy} which was discussed earlier. For instance, the Bernoulli shift with Kolmogorov entropy $h > 0$ has exponential scaling entropy $\scfunclass_{\exp} (X,\mu,T) = [e^{(1-\eps) n h}]$. On the other hand, the usual scaling entropy $\scfunclass(X,\mu,T) = [n]$ does not provide the exact value of Kolmogorov entropy. Similarly, the exponential scaling entropy of any transformation with positive entropy $h$ is the class $[e^{(1-\eps) n h}]$. For transformations with infinite entropy, $\scfunclass_{\exp} = [e^{\frac{n}{\eps}}]$.

This example also shows that even for the Bernoulli shift, the class $\scfunclass_{\exp}(X,\mu,T)$ does not contain functions that are independent of $\eps$. On the other hand, for transformations with pure point spectrum, the class $\scfunclass_{\exp}$ consists of bounded functions and contains $\Phi(\eps, n) = 1$. It is natural to assume that this is the only possible case where the class $\scfunclass_{\exp}$ is stable. The "non-stability" of exponential scaling entropy makes it harder to compute.

The case of positive Kolmogorov entropy shows that the exponential scaling entropy can provide an efficient refinement of the regular scaling entropy in the stable case. However, in the general case, this refinement may be insignificant or even not presented at all. If for a non-stable system $(X,\mu, T)$, for $\Phi \in \scfunclass(X,\mu, T)$, for any $\eps > 0$ there exists $\delta > 0$ such that  $\Phi(\eps, n) =o(\Phi(\delta, n))$, then the class $\scfunclass_{\exp}$ is completely determined by the class $\scfunclass$, since $[\exp(\Phi(\eps,n))]$ does not depend on the choice of a representative $\Phi$ in~$\scfunclass$.

Establishing properties of scaling entropy in the exponential variant is more challenging due to the complexity of its computation. A natural generalization of Theorem~\ref{th_mon_subadd} concerning the monotonicity and subadditivity would be a theorem about the monotonicity and submultiplicativity of exponential scaling entropy. The proof of the monotonicity of scaling entropy remains unchanged in the exponential case.
\begin{theorem}
For any system $(X,\mu,T)$, in the class $\scfunclass_{\exp}(X,\mu,T)$ of exponential scaling entropy, one can always find a function $\Phi \colon \mathbb{R}_+ \times \mathbb{N} \to \mathbb{R}_+$ with the following properties:
 \begin{enumerate}
     \item $\Phi(\,\cdot\,, n)$ is non-increasing for every $n \in \mathbb{N}$;
     \item $\Phi(\eps, \,\cdot\,)$ is non-decreasing for every $\eps>0$.
 \end{enumerate}
\end{theorem} 
Let us also note that not every function $\Phi$ increasing and submultiplicative in $n$ (and decreasing in $\eps$)  (more precisely, the asymptotic class of this function) can be obtained as the exponential scaling entropy of a measure-preserving transformation. An example of such a function is $\Phi(\eps, n) = e^{n + (1-\eps)n^{1/2}}$. Indeed, exponential growth of $\scfunclass_{\exp}(T)$ implies positive Kolmogorov entropy $h$ of the automorphism $T$, and therefore $\scfunclass_{\exp}(T)$ is the class $[e^{(1-\eps) n h}] \not= [\Phi]$. Thus, the direct analog of Theorem~\ref{theorem_complete_description} does not hold for the exponential case. The problem of providing a complete description of possible values of the class~$\scfunclass_{\exp}$ remains open. However, the generic exponential scaling entropy shares the same properties as the generic regular scaling entropy: for any subexponential sequence $\phi(n)$ increasing to infinity, a generic transformation has $\scfunclass_{\exp}$ not comparable to $\phi$.

\subsection{Connections with other invariants}\label{section_slow_entropy} 

In a survey~\cite{KanKWei}, several measure-theoretic invariants similar to scaling entropy are discussed. All of them are effective for automorphisms with zero Kolmogorov entropy. We will mention some of them in the context of their connections to scaling entropy. The \emph{entropy dimension} (Ferenczi--Park, see \cite{FP}) and \emph{slow entropy} (Katok--Thouvenot, see \cite{KT}) are based on the same idea of studying the asymptotics of epsilon-entropy as in scaling entropy. However, the invariant is defined by comparing a growing sequence with a given scale of growing sequences. \emph{Sequential entropy} (Kirillov--Kushnirenko entropy, see \cite{Kush}) is based on the idea of computing entropies of refinements of a partition under the action of a  certain sequence of shifts.
The term "scaled entropy" appears in the paper \cite{ZhPe15}, where it is used for a related but different notion (see also Section 5.3 of the survey~\cite{KanKWei}).

\medskip

{\bf Entropy dimension.}
In the papers by Ferenczi, Park, and other authors, the notion of entropy dimension is introduced and studied.

Let $T \colon (X,\mu) \to (X,\mu)$ be a measure-preserving transformation. For a finite measurable partition $\alpha$ of the space $(X,\mu)$ and a positive $\eps>0$, consider the cut semimetric $\rho_\alpha$ generated by the partition $\alpha$. \emph{The upper entropy dimension} $\overline D(X, \mu,T)$ is defined as follows:
$$
\overline D(\alpha, \eps) = \sup\Big\{s \in [0,1]\colon \limsup_{n \to \infty}\frac{\heps_\eps(X,\mu,T_{av}^n\rho_\alpha)}{n^s}>0 \Big\},
$$
$$
\overline D(\alpha)  = \lim_{\eps \to 0} \overline D(\alpha, \eps),
$$
\begin{equation}\label{eq270202}
\overline D(X, \mu,T)  = \sup_\alpha \overline D(\alpha),
\end{equation}
where the supremum is computed over all finite measurable partitions~$\alpha$.

\emph{The lower entropy dimension} $\underline D(X, \mu,T)$ is defined in a similar way:
$$
\underline D(\alpha, \eps) = \sup\Big\{s \in [0,1]\colon \liminf_{n \to \infty} \frac{\heps_\eps(X,\mu,T_{av}^n\rho_\alpha)}{n^s}>0 \Big\},
$$
$$
\underline D(\alpha)  = \lim_{\eps \to 0} \underline D(\alpha, \eps),
$$
\begin{equation}\label{eq270203}
\underline D(X, \mu,T)  = \sup_\alpha \underline D(\alpha).
\end{equation}
In the case when the upper and lower entropy dimensions coincide, this number is called \emph{the entropy dimension} of the system. The following analog of the Kolmogorov--Sinai theorem holds: the supremums in formulas \eqref{eq270202} and \eqref{eq270203} are attained on generating partitions

It is easy to see that the upper and lower entropy dimensions can be computed given the scaling entropy that is the equivalence class $[\heps_\eps(X,\mu,T_{av}^n\rho)]$ for a generating semimetric $\rho$. However, the converse is not true. The entropy dimension indicates where on the scale of power functions the scaling entropy is located. For more details about the properties of the entropy dimension, we refer the reader to section 5.4 of the survey~\cite{KanKWei} and the referenced papers therein.

{\bf Slow entropy.}
In the papers by Katok and Thouvenot \cite{KT}, the following definition of slow entropy is introduced.

Let $\mathbf{a} = {a_n(t)}_{n \geq 0, t>0}$ be a family of positive increasing sequences tending to infinity, increasing in $t$ (the scale).

\emph{The upper slow entropy} of a system $(X,\mu,T)$ with respect to the scale $\mathbf{a}$ is defined as follows. For a finite measurable partition $\alpha$ of the space $(X,\mu)$ and $\eps>0$, consider the corresponding cut semimetric $\rho_\alpha$ and the set
$$
\overline B(\eps, \alpha) = \Big\{t>0 \colon \limsup_{n\to \infty}  \frac{\exp(\heps_\eps(X,\mu,T_{av}^n\rho_\alpha))}{a_n(t)}>0\Big\}\cup \{0\}.
$$ 
Then define
$$
\overline{ent}_{\mathbf{a}}(T,\alpha) = \lim_{\eps \to 0} \, \sup \overline B(\eps,\alpha),
$$
$$
\overline{ent}_{\mathbf{a}}(X,\mu,T) = \sup_\alpha \overline{ent}_{\mathbf{a}}(T,\alpha),
$$
where the supremum is computed over all finite measurable partitions~$\alpha$.

The quantity $\overline{ent}_{\mathbf{a}}(X,\mu,T)$ is called \emph{the upper slow entropy} of the system $(X,\mu,T)$ with respect to the scale $\mathbf{a}$.

Similarly, \emph{the lower slow entropy} $\underline{ent}_{\mathbf{a}}(X,\mu,T)$ is defined as follows.
$$
\underline B(\eps, \alpha) = \Big\{t>0 \colon \liminf_{n\to \infty } \frac{\exp(\heps_\eps(X,\mu,T_{av}^n\rho_\alpha))}{a_n(t)}>0\Big\}\cup \{0\},
$$ 
$$
\underline{ent}_{\mathbf{a}}(T,\alpha) = \lim_{\eps \to 0} \, \sup \underline B(\eps,\alpha),
$$
$$
\underline{ent}_{\mathbf{a}}(X,\mu,T) = \sup_\alpha \underline{ent}_{\mathbf{a}}(T,\alpha).
$$

Slow entropy is directly related to exponential scaling entropy. By comparing the class $\scfunclass_{\exp}$ with the scale $\mathbf{a}$, one can compute the values $\overline{ent}_{\mathbf{a}}$ and $\underline{ent}_{\mathbf{a}}$. Conversely, knowing the values of slow entropy for all possible scales $\mathbf{a}$ allows one to distinguish between systems with different $\scfunclass_{\exp}$ classes. Thus, slow entropy (more precisely, the collection of slow entropies with respect to all possible scales~$\mathbf{a}$) distinguishes the same dynamical systems as exponential scaling entropy. However, no countable set of scales $\mathbf{a}$ is sufficient to fully recover the class $\scfunclass_{\exp}$. 

Section 4 of the survey~\cite{KanKWei} is dedicated to various properties of slow entropy and examples of its computation.

\medskip
{\bf Sequential Entropy.}
Sequential entropy was introduced by Kushnirenko in~\cite{Kush}. Let $A = \{a_k\}_{k=1}^\infty$ be a given increasing sequence of natural numbers. The entropy $h_A(T)$ of the automorphism $T$ on the space $(X,\mu)$ is defined as follows:
$$
h_A(T,\alpha) = \limsup_{k \to \infty} \frac{1}{n}H\big(\bigvee_{j=1}^k T^{-a_{j}}\alpha\big),
$$
$$
h_A(T) = \sup_{\alpha}h_A(T,\alpha),
$$
where the supremum is computed over all finite measurable partitions~$\alpha$.

It was proved in~\cite{Kush} that an automorphism $T$ has a pure point spectrum if and only if $h_A(T) = 0$ for any sequence $A$. Comparing this result with Theorem \ref{th_diskr_spectr}, we conclude that the boundedness of the scaling entropy is equivalent to this condition.

It turns out that in a certain "neighborhood" of the pure point spectrum, there exist two-sided estimates that relate scaling entropy to sequential entropy.

\begin{theorem}
For any increasing sequence of integers $A$, there exists an unbounded increasing sequence $h = \{h_n\}_{n}$ such that for any automorphism $T$ of the space $(X,\mu)$, if the relation $\scfunclass(T,X,\mu) \prec h$ holds, then $h_A(T) = 0$.

Conversely, for any unbounded increasing sequence $h = \{h_n\}_{n}$, there exists a sequence $A$ such that if $h_A(T) = 0$  then $\scfunclass(T,X,\mu) \prec h$.
\end{theorem}

For more details on the properties of sequential entropy, refer to Section 3 of the survey~\cite{KanKWei}.

\bigskip
\bigskip
{\bf Acknowledgment.} 
The authors are grateful to an anonymous referee for the useful remarks and comments. We also thank Natalia Tsilevich for her help in translating the manuscript into English.

\appendix

\chapter{Several proofs}

\section{Proofs of Lemmas~\ref{lemma_Kepsentropy1} and \ref{lemma_Kepsentropy2}}
\label{App_Kepsentropy_proofs}
\begin{proof}[Proof of Lemma~\ref{lemma_Kepsentropy1}]
Let $n = \exp\Big(\heps_\delta (X,\mu,\rho)\Big)$, and let $X = X_0\cup X_1\cup\dots \cup X_n$ be the corresponding partition: $\mu(X_0)<\delta$, $\diam_\rho(X_j)<\delta$ for $j=1,\dots, n$.
Choose an arbitrary point $x_0 \in X$ and points $x_j \in X_j$ for $j=1,\dots, n$. Define the discrete measure $$
\nu = \sum_{j=0}^n \mu(X_j)\delta_{x_j}
$$
and the transport plan $\gamma$ that transports the sets $X_j$ to $x_j$. Obviously, $\gamma$ is a pairing of measures $\mu$ and $\nu$, and we have 
$$
\int_{X \times X} \rho\, d\gamma = \sum_{j=0}^n \int_{X_j}\rho(x,x_j)\, d\mu(x) \leq \int_{X_0}\rho(x,x_0)\, d\mu(x) + \delta. 
$$
The average value of the right-hand side, when we choose $x_0 \in X$ randomly according to the measure $\mu$, is
$$
\int_{X_0\times X}\rho\, d(\mu\times\mu) + \delta < \eps
$$
due to condition~\eqref{eq11}. Therefore, with an appropriate choice of $x_0$, we can obtain the inequality
$$
d_K(\mu,\nu) \leq \int_{X \times X} \rho\, d\gamma < \eps.
$$
It remains to observe that $H(\nu) \leq \log(n+1)$.
\end{proof}

Before proving Lemma~\ref{lemma_Kepsentropy2}, let us prove the following auxiliary lemma.
\begin{lemma}\label{lement} 
Let $P =(p_j)_{j=1}^N$ be a finite probability vector, and $H(P) = \sum_{j=1}^N p_j \log(1/p_j)$ be its entropy. Let $\delta\in (0,1)$, and $F = \exp\left(\frac{H(P)+1}{\delta}\right)$. Then there exists a subset $J \subset \{1,\dots, N\}$ such that 
$$
|J| \leq F, \qquad \sum_{j \in J}p_j \geq 1-\delta.
$$  
\end{lemma}
\begin{proof}
Without loss of generality, we can assume that the numbers $p_j$ decrease: $p_1\geq p_2\geq \dots \geq p_N$. Let $m$ be the smallest number for which 
$$
\sum_{j=1}^m p_j \geq 1-\delta.
$$
Suppose that the statement of the lemma is false; then $m>F$. Consequently, $1/F> p_m \geq p_{m+1}$. Thus, 
$$
\delta \geq \sum_{j=m+1}^{N} p_j > \delta - 1/F.
$$
Therefore,
$$
H(P) > \sum_{j=m+1}^{N} p_j \log(1/p_j) \geq  \sum_{j=m+1}^{N} p_j  \log(1/p_{m+1}) \geq (\delta - 1/F)  \log F \geq \delta \log F - 1,
$$
which contradicts the definition of $F$.
\end{proof}

\begin{proof}[Proof of Lemma \ref{lemma_Kepsentropy2}]
Let $\nu$ be a discrete measure such that $d_K(\mu,\nu)<\eps^2$. Using Lemma~\ref{lement} for the distribution of $\nu$ and $\delta = \eps$, we find $m \leq \exp\left(\frac{H(\nu)+1}{\eps}\right)$ and a set $M = \{x_1,\dots, x_m\} \subset X$ for which $\nu(M) \geq 1-\eps$. 

Let $\pi_1$ and $\pi_2$ be the projections from $X\times X$ onto the first and second factors, respectively. Let $\gamma$ be an optimal transport plan for the measures $\mu$ and~$\nu$, that is a probability measure on $X\times X$ such that $\pi_1(\gamma) = \mu$, $\pi_2(\gamma) = \nu$, and
$$
\int_{X\times X}\rho \,d\gamma = d_K(\mu,\nu) < \eps^2.
$$
Consider the set $A = \big\{(x,y) \in X\times X \colon \rho(x,y)< \eps\big\}$. By Chebyshev's inequality, $\gamma(A)\geq  1-\eps$. Notice that 
$\gamma(\pi_2^{-1}(M)) = \nu(M) \geq 1-\eps$,
thus 
$$
\gamma(A\cap \pi_2^{-1}(M)) \geq 1-2\eps,
$$ 
and consequently 
$$
\mu\Big(\big\{x \in X \colon \rho(x,M)<\eps\big\}\Big) \geq \gamma\Big(\big\{A\cap \pi_2^{-1}(M)\big\}\Big) \geq 1-2\eps.
$$
Therefore, 
$$
\heps_{2\eps}(X,\mu,\rho) \leq \log|M| \leq \frac{H(\nu)+1}{\eps}.
$$
By minimizing the right-hand side over discrete measures $\nu$ that satisfy the inequality $d_K(\mu,\nu)<\eps^2$, we arrive at the inequality~\eqref{eq10}.
\end{proof}

\section{Proof of Theorems~\ref{th_two_metrics2} and \ref{th_comp_triples}}
\label{App_th_triples}

Theorem~\ref{th_two_metrics2} is a consequence of the lemma presented below.
\begin{lemma}
1. Let $\rho_1,\rho_2$ be two measurable semimetrics on $(X,\mu)$. Then, for any $\delta>0$, there exists a semimetric space $(Y,\rho)$ and isometries $\phi_1\colon (X,\rho_1) \to (Y,\rho)$, $\phi_2\colon (X,\rho_2) \to (Y,\rho)$, such that 
$$d_K\big(\phi_1(\mu_1),\phi_2(\mu_2)\big) \leq \|\rho_1-\rho_2\|_{m} +\delta.$$

2. Let $\mu_1$ and $\mu_2$ be two probability measures on a semimetric space $(Y,\rho)$. Then, there exists a probability space $(X,\mu)$ and measurable maps $\psi_1\colon X \to Y$ and $\psi_2\colon X \to Y$, such that $\psi_1(\mu) = \mu_1$, $\psi_2(\mu) = \mu_2$, and 
$$
\big\|\rho\circ\psi_1 - \rho\circ\psi_2\big\|_{m} \leq 2 d_K(\mu_1,\mu_2).
$$
\end{lemma}
\begin{proof}
1. Let $D$ be a semimetric on $(X,\mu)$ such that 
$$
|\rho_1(x,y)- \rho_2(x,y)| \leq D(x,y) \quad \text{a.\,e.}
$$
and $\int_{X^2} D \ d\mu^2 \leq \|\rho_1-\rho_2\|_{m} +\delta$.
Find a function $f$ on $(X,\mu)$ such that
\begin{equation}\label{eq1}
D(x,y) \leq f(x) + f(y)
\end{equation}
and 
\begin{equation}\label{eq2}
\int_X f d\mu \leq \int_{X^2} D\ d\mu^2.
\end{equation}
In order to satisfy~\eqref{eq1} we can choose $f = D(\,\cdot\,, x_0)$ for any $x_0 \in X$. By selecting $x_0$ in such a way that the integral of the function $D(\,\cdot\,, x_0)$ is minimized, we obtain~\eqref{eq2}.

Consider the set $Y = X \times \{1,2\}$ and define a semimetric $\rho$ on $Y$ as follows:
\begin{align*}
\rho\big((x,i),(y,i)\big) &= \rho_i(x,y), \quad i=1,2, \quad x,y \in X,\\
\rho\big((x,2),(y,1)\big) & = \essinf\limits_{z \in X} (\rho_2(x,z) + f(z) + \rho_1(z,y)), \quad x,y \in X.
\end{align*}
Define isometries $\phi_i\colon (X,\rho_i) \to (Y,\rho)$, $i=1,2$, by $\phi_i(x) = (x,i)$, $x \in X$. The Kantorovich distance $d_K$ between the images of the measures $\mu_i$ under the mappings $\phi_i$ does not exceed $\int_X f d\mu \leq \|\rho_1-\rho_2\|_{m} +\delta$. 

2. Let $\mu$ be a measure on $X = Y^2$ that is a coupling of $\mu_1$ and $\mu_2$ realizing the Kantorovich transport plan between these measures:
$$
\int_{Y^2} \rho\ d\mu = d_K(\mu_1,\mu_2). 
$$
Let $\psi_1$ and $\psi_2$ be the projections of $X=Y^2$ onto the first and second factors. Then, the semimetrics $\rho_i = \rho\circ \psi_i$, $i=1,2$, on $X$ are given by the formulas
$$
\rho_i\big((x_1,x_2),(y_1,y_2)\big) = \rho(x_i,y_i).
$$
Define a semimetric $D$ on $X$ as follows:  
$$
D\big((x_1,x_2),(y_1,y_2)\big) = \rho(x_1,x_2) + \rho(y_1,y_2), \quad (x_1,x_2) \ne (y_1,y_2).
$$
Then $|\rho_1 - \rho_2| \leq D$, hence
$$
\big\|\rho_1 - \rho_2\big\|_{m} \leq \int_{X^2} D\ d\mu^2 = 
2 \int_{Y^2} \rho\ d\mu = 2 d_K(\mu_1,\mu_2).\qedhere
$$
\end{proof}

\begin{proof}[Proof of  Theorem~\ref{th_comp_triples}]

In one direction, the statement of the theorem can be proved quite easily. Let us assume that the set $M$ is precompact. Fix $\eps>0$ and find a finite subset $J\subset I$ such that the triples $\{(X_j,\mu_j,\rho_j)\colon i \in J\}$ form a finite $\frac{\eps^2}{32}$-net in the metric $\mdist$. For each $i \in I$ find $j \in J$ such that
$$
\mdist\Big((X_i,\mu_i,\rho_i),(X_j,\mu_j,\rho_j)\Big) < \frac{\eps^2}{32}.
$$
Therefore, there exists a coupling $(X,\mu)$ and projections $\psi_{i,j}\colon (X,\mu) \to (X_{i,j},\mu_{i,j})$ such that 
\begin{equation}
    \big\|\rho_i\circ \psi_i - \rho_j\circ \psi_j\big\|_m < \frac{\eps^2}{32}.
\end{equation}
Then
$$
\heps_\eps(X_i,\mu_i,\rho_i)=
\heps_\eps(X,\mu,\rho_i\circ \psi_i)\leq
\heps_{\eps/4}(X,\mu,\rho_j\circ \psi_j)=
\heps_{\eps/4}(X_j,\mu_j,\rho_j).
$$
Hence,
$$
  \sup_{i \in I} \heps_\eps(X_i,\mu_i, \rho_i) \leq 
\max_{j \in J} \heps_{\eps/4}(X_j,\mu_j, \rho_j)  <\infty, 
$$
and Condition~2 is proved.

Let us prove the uniform integrability. For $i \in I$ and the corresponding $j \in J$, let $d_i$ be a semimetric on $(X,\mu)$ such that
\begin{equation}\label{eq12}
    \rho_i\circ \psi_i \leq \rho_j\circ \psi_j + d_i, \qquad \int_{X^2} d_i \, d\mu^2 < \frac{\eps^2}{32}.
\end{equation}
For any $R>0$, let
$L_{i,j,R} = \{\rho_i\circ \psi_i > 2R\} \subset X^2$.
Then
\begin{multline}
\mu^2(L_{i,j,R}) \leq \mu^2(\{\rho_j\circ \psi_j > R\})+ \mu^2(\{d_i > R\}) \leq\\
\leq \frac{1}{R} \Big(\int_{X^2}\rho_j\circ \psi_j \, d\mu^2 + \frac{\eps^2}{32} \Big)
=\frac{1}{R} \Big(\int_{X_j^2}\rho_j \, d\mu_j^2 + \frac{\eps^2}{32} \Big).
\end{multline}
Let $Q(\eps) = \max_{j\in J} \int_{X_j^2}\rho_j \, d\mu_j^2 + \frac{\eps^2}{32}$. Then we obtain
\begin{align}
\int_{\rho_i>2R} \rho_i\, d\mu_i^2=&\!\!\!\!
\int_{L_{i,j,R}}\!\! \rho_i\circ \psi_i \,d\mu^2 \leq \!\!\!\!
\int_{L_{i,j,R}}\!\! \rho_j\circ \psi_j \,d\mu^2 + \frac{\eps^2}{32} \leq \\
\leq&
\sup\Big\{\int_{L} \rho_j \, d\mu_j^2\colon L \subset X_j^2, \mu_j^2(L)\leq \frac{Q(\eps)}{R}\Big\}+ \frac{\eps^2}{32}. 
\end{align}
The right-hand side of the last inequality converges to $\frac{\eps^2}{32}$ as $R$ tens to infinity (for each $j$, thus uniformly over the finite set $J$). Therefore,
$$
\limsup_{R \to \infty}\sup_{i \in I} \int_{\rho_i>2R} \rho_i \, d\mu_i^2 \leq 
\frac{\eps^2}{32}.
$$
The left-hand side of the obtained inequality does not depend on $\eps$, thus it is equal to zero. Condition 1 of uniform integrability is thus proved.

The proof of the theorem in the reverse direction is more complicated and involves more reasoning. Suppose that Conditions 1 and 2 are satisfied. For each $\eps>0$, we want to find a finite $\eps$-net in the set $M$ with respect to the metric $\mdist$. Utilizing the condition of uniform integrability, we can choose $R>0$ such that $\int_{\rho_i>R}\rho_i \,d\mu_i^2<\eps/2$ for all $i \in I$. Consequently,
$$
\|\rho_i - \min(\rho_i,2R)\|_m\leq \int_{\rho_i>R}\rho_i \,d\mu_i^2<\eps/2.
$$
Thus, it suffices to find a finite $\eps/2$-net with respect to the metric $\mdist$ within the set of triples $\{(X_i,\mu_i,\min(\rho_i,2R))\}_{i \in I}$. Therefore, without loss of generality, we can  assume that the semimetrics $\rho_i$ are uniformly bounded. Due to homogeneity, one can assume that all $\rho_i$ do not exceed 1.

\begin{lemma}\label{lem11}
Let $\eps>0$ be fixed, and $\sup_{i\in I}\heps_\eps(X_i,\mu_i,\rho_i) <\infty$. Then on a standard probability space $(X,\mu)$, one can find a finite partition $\xi = (A_1,\dots, A_n)$ with the following property: for each $i \in I$, there exists a homomorphism of measure spaces $\psi_i \colon (X,\mu) \to (X_i,\mu_i)$ and a set $B_i \in X$, $\mu(B_i)<2\eps$, such that the sets $A_k\setminus B_i$, $k=1,\dots, n$, have  diameters less than $\eps$ in the semimetric~$\rho_i\circ \psi_i$.
\end{lemma}
\begin{proof}
Let $N$ be such that $\heps_\eps(X_i,\mu_i,\rho_i) \leq \log(N)$ for all $i \in I$. For each $i \in I$, there exists a measurable finite partition $\xi^i = \{A_1^i,\dots,A_N^i\}$ of the space $(X_i,\mu_i)$ and a set $E^i \subset X_i$, $\mu_i(E^i)< \eps$, such that diameters of the sets $A_j^i \setminus E^i$ in the semimetric $\rho_i$ are smaller than $\eps$. Let $P^i = (\mu_i(A_k^i))_{k=1}^N$ be the probability vector corresponding to the partition $\xi^i$.

Let $P = (p_1,\dots, p_N)$ be a fixed probability vector of length $N$. On the space~$(X,\mu)$, we choose a partition $\xi_P = (A_1,\dots, A_N)$ into $N$ parts with measures $\mu(A_k)=p_k$. If the vector $P^i=(\mu(A_k^i))_{k=1}^N$ satisfies $\sum_{k=1}^N |p_k - \mu(A_k^i)|<\eps$, then the space $(X_i,\mu_i)$ can be realized on $X$ in such a way that the partition~$\xi^i$ will differ slightly from $\xi_P$. Specifically, one can find a homomorphism of measure spaces $\psi_i\colon (X,\mu) \to (X_i,\mu_i)$ such that 
$$
\sum_{k=1}^N \mu(\psi_i^{-1}(A_k^i)\Delta A_k)<\eps.
$$
Then there exists a set $B_i \subset X$ with $\mu(B_i) < 2\eps$, such that diameters of the sets $A_k\setminus B_i$ in the semimetric $\rho_i \circ \psi_i$ will be less than $\eps$. 

The set of vectors $\{P^i\colon i \in I\}$ is bounded in a finite-dimensional space, thus  within this set we can find a finite $\eps$-net with respect to the $L^1$ metric: let it be $\{P_1, \dots, P_m\}$. For each element $P_j$ of this $\eps$-net, we construct a corresponding partition $\xi_{P_j}$ on $(X,\mu)$. In the role of the desired partition $\xi$, we can take a refinement of the partitions $\xi_{P_j}$. The lemma is proved.
\end{proof}

The rest of the proof of the theorem follows the proof of the corresponding implication of Theorem~\ref{lemmcomp}. Using Lemma~\ref{lem11}, we will work with semimetrics~$\tro_i = \rho_i\circ \psi_i$ on the space $(X,\mu)$ and the partition $\xi = \{A_1,\dots, A_n\}$. We will show that in the set $\{\tro_i\colon i \in I\}$, one can find a finite $7\eps$-net with respect to the $m$-norm. To achieve this, we will demonstrate that these semimetrics can be approximated in the $m$-norm by a bounded set in a finite-dimensional space of semimetrics, that are constant on the sets $A_j\times A_k$, $j,k \in\{1,\dots,n\}$.

Let $\tro \in \{\tro_i\colon i \in I\}$. Find a set $B\subset X$, with $\mu(B)<2\eps$, such that each of the sets $A_j \setminus B$ have diameters less than $\eps$ in the semimetric $\tro$. If needed, by adding a subset of measure zero to $B$, we can assume that each of the sets $A_j\setminus B$ is either empty or of positive measure. Choose a point $x_j$ in each $A_j$ such that the functions $\tro(\,\cdot\,, x_j)$ are measurable on $X$, and if $\mu(A_j \setminus B)>0$, then $x_j \in A_j \setminus B$. We define a semimetric $\bro$ on $X$ as follows: for $x\ne y$ if $x \in A_k$, $y \in A_j$, we set $\bro(x,y) = \tro(x_k,x_j)$. Then, obviously, if $x, y \notin B$, then $$
|\tro(x,y) - \bro(x,y)| < 2\eps.
$$
Since $\bro$ and $\tro$ are pointwise bounded by 1, for all $x,y \in X$, we have the following estimate
$$
|\tro(x,y)-\bro(x,y)|\leq \big(2\eps + \chi_{B}(x)+\chi_B(y)-\chi_{B}(x)\chi_B(y)\big)\chi_{\{x\ne y\}}. 
$$
On the right-hand side of the last inequality, there is a semimetric whose integral does not exceed $2\eps + 2\mu(B) < 6 \eps$, hence
$$\|\tro(x,y)-\bro(x,y)\|_m<6\eps.$$

We have shown that each of the semimetrics $\tro_i$, $i \in I$, can be approximated by the corresponding semimetric $\bro_i$ with an accuracy of up to $6\eps$ in the $m$-norm. Moreover, all the semimetrics $\bro_i$ are contained in a finite-dimensional space and are uniformly bounded, forming a precompact set. Therefore, in the set $\{\tro_i\}_{i \in I}$, we can find a finite $7\eps$-net with respect to the $m$-norm. The theorem is proved.
\end{proof}

\section{Proof of Proposition~\ref{proposition_statsum}}\label{section_proof_statsum}

        To prove the inequality in one direction is quite straightforward:
\begin{equation}
\Omega_T\big(\rho,1 - \frac{1}{n}\big) \ge \frac{1}{e} T_{av}^n\rho.
\end{equation}
Thus,
\begin{equation}
\heps_\eps(X,\mu, \Omega_T\big(\rho,1 - \frac{1}{n}\big)) \ge \heps_{e\eps}(X,\mu,  T_{av}^n\rho).
\end{equation}
Hence, 
\begin{equation}
\tilde\Phi_\rho(\eps, 1 - \frac{1}{n}) \succeq \Phi_\rho.
\end{equation}

To obtain the reverse estimate, for each $\eps > 0$ find a constant $c = c(\eps, \rho)$ such that for any natural $n$,
\begin{equation}
\Big\| (1-z) \sum\limits_{k > cn} z^k T^{-k} \rho \Big\|_m < \eps^2/32,
\end{equation}
where $z = 1 - \frac{1}{n}$. Then, due to Lemma~\ref{lemma_semicontinuity},
\begin{equation}
\heps_\eps\big(X,\mu, \Omega_T(\rho,1 - \frac{1}{n})\big) \le \heps_{\frac{\eps}{4}}\big(X,\mu, (1-z)\sum_{k=0}^{cn} z^n T^{-k}\rho\big) \le 
\heps_{\frac{\eps}{4c}}(X,\mu, T_{av}^{cn} \rho),
\end{equation}
where $z = 1 - \frac{1}{n}$ again. However, due to the subadditivity of scaling entropy (see Section~\ref{section_scentropy_values}), the function $\Psi(\eps, n) = \Phi_\rho(\frac{\eps}{4c}, cn)$ is equivalent to the function~$\Phi_\rho$. Thus, $\tilde\Phi_\rho(\eps, 1 - \frac{1}{n}) \preceq \Phi_\rho$, and therefore, the function $\tilde\Phi_\rho(\eps, 1 - \frac{1}{n})$ belongs to the class $\scfunclass(T)$.

\section{Proof of Proposition~\ref{ergodic_decomosition}}\label{section_proof_erg_dec}

Let $E$ be a set of positive measure such that for any $\alpha \in E$, $\mathcal{H}(X,\mu_\alpha,T)$ grows not slower than $h$.

Suppose the contrary. Then there exists a subsequence $n_j$ satisfying the relation $\Phi(\eps, n_j) \prec h_{n_j}$ for any $\Phi \in \mathcal{H}(X, \mu, T)$ and any $\eps > 0$. Consider an admissible metric $\rho$ on $X$ and a certain index $n_j$. Let $X_0, X_1, \ldots X_k$ be sets that realize the $\eps$-entropy of the triple $(X, \mu, T_{av}^{n_j} \rho)$. Then
$$
\eps > \mu(X_0) = \int\mu_\alpha(X_0) d\nu(\alpha).
$$
It is clear that there exists a constant $r > 0$, depending only on $\nu(E)$, such that on some set $E_j \subset E$ of measure $r$, the inequality $\mu_\alpha(X_0) < \frac{\eps}{r}$ holds. For such $\alpha$, the following inequality holds:
$$
\heps_\frac{\eps}{r}(X, \mu_\alpha, T_{av}^{n_j} \rho) \leq \heps_\eps(X, \mu, T_{av}^{n_j} \rho).
$$

The measure of those $\alpha$ for which $\alpha \in E_j$ infinitely many times is positive. Choose such an $\alpha$ and a subsequence of indices $j_m$ for which $\alpha \in E_{j_m}$. We obtain
$$
h_{n_{j_m}} \preceq \heps_\frac{\eps}{r}(X, \mu_\alpha, T_{av}^{n_{j_m}} \rho)  \leq \heps_\eps(X, \mu, T_{av}^{n_{j_m}} \rho) \prec h_{n_{j_m}}.
$$
This is a contradiction.

\chapter{Some open questions}\label{section_problems}

We list several problems related to the new concepts referred to in the survey, without aiming for a comprehensive coverage of the topics. Some of these problems have already been mentioned in the main text.

\medskip
{\bf 1. Theory of mm-spaces, classification of metric triples, and problems about matrix distributions.}
Perhaps the most important general question is as follows: to what extent does matrix distribution, as a measure on the space of distance matrices, allow one to describe various properties of mm-spaces (spaces with measure and metric)?

On one hand, as we have seen, matrix distribution is a complete invariant up to measure-preserving isometries. However, the practical use of it as a tool for studying spaces needs to be developed. In particular, the question of what can be said about the random spectra of matrix distributions for the most natural mm-spaces is still open, even though it was posed a long time ago (see Section~\ref{sec_spect}). A similar question can be posed for metric triples: can a metric triple, particularly the metric, be reconstructed from the asymptotics of random spectra of consecutive minors of the matrix distribution? Most likely, the answer is negative, but it is interesting to explore which properties of a triple are ``spectral,'' i.\,e., depending only on the spectrum. Here, it is appropriate to recall extensive literature on spectral geometry of graphs, metric spaces, etc. However, the task mentioned above is fundamentally different, as we consider random spectra, i.\,e., stochastic, rather than individual characteristics of sets of minors' eigenvalues. This statistics is fundamentally different from the statistics of random Gaussian matrices, i.\,e., semi-circular and similar laws.

The limiting distributions of spectra for the most natural manifolds are also interesting and currently unknown. These include spheres, Stiefel  manifolds, as well as non-compact manifolds with probability measures, see the forthcoming paper~\cite{VP23}.

\medskip
{\bf 2. Calculation of scaling entropy.}
A significant part of the survey is devoted to the relatively new concept of scaling entropy. However, we do not yet know how to compute it even in the most natural cases. It is important to keep in mind that the unbounded growth of scaling entropy indicates the presence of a continuous part in the spectrum of an automorphism, which can sometimes be very difficult to establish directly. This question was posed by A.~Vershik for adic automorphisms, specifically for the {\it Pascal automorphism}, one of the first nontrivial examples of adic transformations defined in the late 1970s by A.~Vershik (see~\cite{V81,V82,VPasc}). Later, it was revealed that this transformation had been used for a partition problem by S.~Kakutani, see~\cite{Kak76}. An attempt to compute the scaling entropy for the Pascal automorphism was made in~\cite{LMM10}. However, the unbounded growth of scaling entropy has not yet been proven, despite the efforts of many mathematicians. Thus, it is not proved that the Pascal automorphism has a purely continuous spectrum. The confidence in the validity of this fact is expressed in the title of~\cite{VPasc}.

On the other hand, calculations have been carried out for the Morse and Chacon transformations, i.\,e., that are substitutions (which are stationary adic shifts on infinite graphs, see~\cite{VL92}), see the details in Section~\ref{Sec_podst}. As a generalization of the result for the Morse transformation, it is interesting to find the scaling entropy for more general skew products over transformations with discrete spectra.

The scaling entropy has not yet been computed for numerous adic transformations on graph paths (Young, Fibonacci, etc.). It is also of interest to compute the scaling entropy of Gaussian automorphisms with simple singular or purely singular spectra (see~\cite{Gir} and~\cite{V62}). Surprisingly, the technique of approximations (ranks) does not help in this case yet.

\medskip

{\bf 3. Description of an unstable automorphism.}
Of undoubted interest is the description of unstable automorphisms, see Section~\ref{Sec321}. In particular, it is of interest to provide an example of a symbolic model for some unstable automorphism.

\medskip
{\bf 4. Development of the theory of scaling entropy for countable groups.}
The theory of scaling entropy for actions of amenable groups described in Section~\ref{section_groups} requires further development. For non-amenable groups, practically nothing is known beyond the definition itself; it is even unknown whether the introduced invariant is nontrivial. In particular, an undoubtedly interesting question is the relationship between the definition of scaling entropy presented in Section~\ref{section_groups} and other entropy definitions for such groups.

The scaling entropy growth gap phenomenon described in Section~\ref{section_entropy_gap} seems to be important and interesting for us. It would be intriguing to understand for which amenable groups this phenomenon occurs. In particular, is it valid for lamplighter groups with non-abelian groups of lamps or for finitely generated simple amenable groups?

\medskip
{\bf 5. Non-Bernoulli automorphisms with completely positive entropy.}
For non-Bernoulli automorphisms with completely positive entropy (also called K-automorphisms), which were introduced by D.~Ornstein in the 1970s, there are still no visible invariants. These invariants should be related to non-entropy asymptotic invariants of stationary metric compact sets. As of now, such invariants are unknown.

\medskip
{\bf 6. Catalytic and relative invariants.}
The scheme for constructing catalytic absolute or relative invariants described in Section~\ref{section_catalyst} (see also Section~\ref{Sec313}) has been realized so far only in the form of scaling entropy. In this case, the invariant is absolute (i.\,e., it does not depend on the metric). There are no other examples currently. This is explained by the fact that, besides epsilon entropy, we lack a developed theory of invariants of compact metric spaces or mm-spaces themselves. 
In particular, there are no invariants of metric compacts equipped with some symmetry (for instance, invariant under automorphisms). There is no reason to doubt the existence of such invariants. This is indicated by the mentioned non-Bernoulli systems with positive Kolmogorov entropy.

\end{document}